\documentclass{arxiv}

\usepackage[sort]{natbib}
\usepackage{mathtools}
\mathtoolsset{showonlyrefs}
\usepackage{amsmath}
\usepackage{amsfonts}
\usepackage{amsthm}
\usepackage{enumitem}
\usepackage{booktabs}

\usepackage{pgfplots,tikz}
\usetikzlibrary{plotmarks}

\newcommand{\e}{\mathrm{e}}
\newcommand{\ri}{\mathrm{i}}
\renewcommand{\d}{\mathrm{d}}
\newcommand{\C}{\mathbb{C}}
\newcommand{\R}{\mathbb{R}}
\DeclareMathOperator{\rank}{rank}
\renewcommand{\Re}{\operatorname{Re}}
\renewcommand{\Im}{\operatorname{Im}}
\newcommand{\funF}{\mathfrak{F}}
\newcommand{\funE}{\mathfrak{E}}
\newcommand{\T}{\mathsf{T}}
\newcommand{\cM}{\mathcal{M}}
\renewcommand{\tilde}{\widetilde}
\renewcommand{\hat}{\widehat}

\newtheorem{theorem}{Theorem}[section]
\newtheorem{definition}{Definition}[section]

\newtheorem{lemma}[theorem]{Lemma}
\newtheorem{remarks}{Remark}[section]

\begin{document}

\articletype{Research Paper} 

\title{Generalized rational Prony and Bernoulli methods}

\author{Tam\'as D\'ozsa$^{1,2}$\orcid{0000-0003-0919-4385}, Matthias Voigt$^1$\orcid{0000-0001-8491-1861}, Zolt\'an Szab\'o$^{2}$\orcid{0000-0001-6183-7603}, J\'ozsef Bokor$^{2}$\orcid{0000-0002-3461-4807} and P\'eter Kov\'acs$^{3}$\orcid{0000-0002-0772-9721}}

\affil{$^1$ UniDistance Suisse, Faculty of Mathematics and Computer Science, Brig, Switzerland}

\affil{$^2$ HUN-REN Institute for Computer Science and Control, Systems and Control Lab, Budapest, Hungary}

\affil{$^3$ E\"otv\"os Lor\'and University, Department of Numerical Analysis, Budapest, Hungary}

\email{tamas.dozsa@unidistance.ch}


\begin{abstract}
The generalized operator-based Prony method is an important tool for describing signals which can be written as finite linear combinations of eigenfunctions of certain linear operators. On the other hand, Bernoulli's algorithm and its generalizations can be used to recover the parameters of rational functions belonging to finite-dimensional subspaces of $H_2$ Hardy-Hilbert spaces. In this work, we discuss several results related to these methods. We discuss a rational variant of the generalized operator-based Prony method and show that in fact, any Prony problem can be treated this way. This realization establishes the connection between Prony and Bernoulli methods and allows us to address some well-known numerical pitfalls. Several numerical experiments are provided to showcase the usefulness of the introduced methods. These include problems related to the identification of time-delayed linear systems and parameter recovery problems in reproducing kernel Hilbert spaces. 
\end{abstract}

\section{Introduction}
\label{sec:intro}

Consider a vector space $V$ over the field of complex numbers $\mathbb{C}$. Let $f \in V$ obey
\begin{equation}
    \label{eq:PronySig}
    f := \sum_{k=1}^M c_k v_{\lambda_k},
\end{equation}
where $\lambda_k, c_k \in \mathbb{C}$ and  $v_{\lambda_k} \in V$ for $k=1,\ldots,M$ and $M \in \mathbb{N} \setminus \{0\}$. Assume that the vectors $v_{\lambda_k}$ for $k=1,\ldots,M$ are linearly independent and are completely defined by $\lambda_k$ in the sense that there exists a $\mathbb{C} \to V$  bijection which maps $\lambda_k$ to $v_{\lambda_k}$. It follows that $\lambda_k \neq \lambda_j \ (k \neq j)$. Without loss of generality, assume that $c_k \neq 0$ for all $k=1,\ldots,M$. We shall refer to $c_k$ as linear and to $\lambda_k$ as the nonlinear parameters, as in the latter case the linear dependence of $v_{\lambda_k}$ on $\lambda_k$ is not assumed. In this work, our objective is as follows. Given $f$, we would like to recover the linear parameters $c_k$ and the nonlinear parameters $\lambda_k$. Following~\cite{peter2013generalized, peter2013generalizedDissertation, stampfer2019generalized, stampfer2020generalized}, we shall refer to $f$ as a \emph{Prony signal} and to  $v_{\lambda_k}$ as \emph{Prony atoms} henceforth.

The parameter recovery of Prony signals has many practical applications. For example, finding the poles and residues which define the (rational) transfer function of a single-input single-output linear time-invariant (SISO LTI) dynamical system can be posed as such a problem~\cite{dozsa2025system}. Another example application is the sparse reconstruction of signals using orthogonal polynomial bases~\cite{stampfer2020generalized}. Variants of Prony's method have also been successfully applied to various medical imaging problems~\cite{jaramillo2014improving, jani2023prony}.

This study focuses on generalized Prony and Bernoulli methods as detailed in~\cite{siamdozsa} and~\cite{SS1}. A ``general" description of the strategy used by Prony-like algorithms can be given as follows. The recovery of $c_k$ and $\lambda_k$ are performed in two separate steps, where in the first step, the nonlinear parameters are identified. In the case of Prony's method, this usually involves finding the solution to a linear system of equations characterized by a Hankel matrix (see section~\ref{sec:genPro} and~\cite{peter2013generalizedDissertation, stampfer2019generalized}). It should be noted however, that many other strategies exist, with various benefits and drawbacks. Of these we shall focus on Bernoulli's method in this study, however we note that matrix pencil methods~\cite{ikramov1993matrix}, the ESPRIT algorithm~\cite{roy1986estimation}, the quotient-difference (QD) algorithm~\cite{tsypin2020application}, and several alternatives are known. For a thorough overview of Prony algorithms and related methods, the authors recommend~\cite{peter2013generalized} and~\cite{stampfer2019generalized}. In the second step of a Prony-like scheme, the linear parameters $c_k$ are recovered. This usually involves solving a Vandermonde system. Although numerically stable solutions to this have been developed~\cite{higham1988fast}, in this study we propose a simple alternative based on rational orthogonal expansions.

The classical Prony method, originally developed by Gaspard de Prony in 1795~\cite{prony1795} can be viewed as a special case of the problem class posed in Eq.~\eqref{eq:PronySig}. Let $\ell$ denote the vector space of all complex sequences $\mathbb{N} \cup \{0\} \to \mathbb{C}$. In particular, choosing $V := \ell$ and $v_{\lambda_k} := \left( \e^{\lambda_k}, \e^{2 \lambda_k}, \e^{3 \lambda_k}, \ldots \right) \in \ell$ for $\lambda_k \in \mathbb{R} + \ri[0, 2\pi)$ for $k=1,\ldots,M$ yields the classical Prony problem.

In this study, we rely heavily on the the generalized Prony method proposed in~\cite{peter2013generalized} by Peter and Plonka. In~\cite{peter2013generalized}, the observation is made, that if the Prony atoms $v_{\lambda_k}$ coincide with the eigenfunctions of an appropriate linear operator, then an algorithm similar to Prony's original method can be utilized to recover the desired parameters. In section~\ref{sec:genPro} we review this construction along with the generalized operator-based Prony (GOP) approach introduced by Stampfer and Plonka in~\cite{stampfer2020generalized}. 
Finally, a prior result which is heavily referenced in this study is~\cite{dozsa2025system}, where the framework proposed in~\cite{peter2013generalized} is applied to Prony signals belonging to finite-dimensional subspaces of the Hardy-Hilbert space $H_2(\mathbb{D})$ of complex functions which are analytic in the open unit disk. As we will discuss in more detail later, these signals correspond to certain transfer functions of single-input single-output linear time-invariant (SISO LTI) systems and are therefore relevant in applications. 
In~\cite{dozsa2025system} the authors show that this problem can be addressed using the GOP framework. We further elaborate on the specifics of the above cited schemes and propose some mild generalization of Prony-like problems in section~\ref{sec:genPro}.

Bernoulli's original method to recover the poles of rational functions is treated, e.g.,  in~\cite{HEN93}. In particular, using the notation introduced for Prony problems, let $V=H_2(\mathbb{D})$, and consider the the Prony signal
\begin{equation}
    f(z) = \sum_{k=1}^M c_k  \frac{1}{1 - \overline{\lambda_k}z} \quad (c_k \in \mathbb{C}, \ \lambda_k \in \mathbb{D}, \ z \in \overline{\mathbb{D}}),
\end{equation}
where $\overline{\mathbb{D}}$ denotes closed complex unit disk. Note that because $f$ is rational, it can be analytically continued from the domain $\mathbb{D}$ to $\overline{\mathbb{D}}$. Denoting the set of nonlinear parameters by $\Lambda = \{\lambda_1, \lambda_2, \ldots, \lambda_M \} \subset \mathbb{D}$, suppose that 
\begin{equation}
    \label{eq:dpole}
    |\lambda_{1}|>|\lambda_k| \quad (k=2, \ldots, M).
\end{equation}
If $\lambda_1$ satisfies Eq.~\eqref{eq:dpole}, it is referred to as a dominant nonlinear parameter and Bernoulli's method can be used to recover it. Bernoulli's method can be applied as long as a unique dominant parameter $\lambda_1$ among the nonlinear parameters $\lambda_k \ (k=1,\ldots,M)$ exists. Indeed, since $f \in H_2(\mathbb{D})$ (see, e.g., \cite[Eq.~(2.10)]{siamdozsa}, or~\cite[Prop.~12.3.4 (b)]{bprodbook}), the Fourier coefficients of $f$ exist and can be written as
\begin{equation}
    \label{eq:fcoeff}
    f_n := \langle f, z^n \rangle_{H_2(\mathbb{D})} = \frac{1}{2 \pi \ri} \int_{\mathbb{T}} f(z) \overline{z^n}\, \d z = \sum_{k=1}^M c_k \overline{\lambda_k}^n \quad (n \in \mathbb{N} \cup \{0\}),
\end{equation}
where $\mathbb{T}$ denotes the unit circle.
Suppose the parameter $\lambda_1$ is dominant as defined in Eq.~\eqref{eq:dpole} and consider the sequence ${(q_n)}_{n \ge 0}$ with $q_n := \frac{f_{n+1}}{f_n} \ (n \in \mathbb{N} \cup \{0\})$. We have
\begin{equation}
    \label{eq:BernOrig}
    q_n := \overline{\lambda_1} + \mathcal{O}(\beta^{n}),
\end{equation}
where $\beta := \max_{k=2,\ldots,M} |\lambda_k| / |\lambda_1|$. A limitation of this scheme is that the sequence $q_n$ fails to converge if there is no dominant nonlinear parameter and convergence can be slow if there is no clear gap between the dominant and the other nonlinear parameters. When considering real-life applications, for example recovering the poles of a transfer function describing a linear system, unfortunately this is often the case~\cite{szidar, vandenhof}. To remedy this issue, Soumelidis and Schipp introduce a generalization of the method relying on Laguerre-Fourier coefficients~\cite{SS1}, which is guaranteed to converge even in the absence of a unique dominant nonlinear parameter. This result is  further generalized in~\cite{siamdozsa}, where the rational orthogonal Takenaka-Malmquist system~\cite{takenakaorig, malmquistorig}, and later, non-orthogonal function systems from $H_2(\mathbb{D})$ sub-algebras~\cite{med24} are utilized to construct the sequence ${(q_n)}_{n \ge 0}$ (for details see also section~\ref{sec:genBer}). One useful property of the generalized Bernoulli (GB) scheme is that the nonlinear parameters $\lambda_k$ can be recovered in an iterative fashion by repeatedly considering Eq.~\eqref{eq:BernOrig}, see, e.g., ~\cite{SBS, SSSAB} and section~\ref{sec:genBer}. Thus, the method is usable when the number of nonlinear parameters $M$ is unknown a priori. Furthermore, it can be used to construct a reduced-order model of $f$, which is guaranteed (at least asymptotically) to posses the same dominant nonlinear parameters as $f$. 

In this study the following novel contributions are described in detail:
\begin{enumerate}
    \item We show that parameter estimation problems that can be solved by the generalized operator-based Prony's method detailed in~\cite{stampfer2020generalized} can always be posed as pole finding tasks. This result is summarized by Theorem~\ref{thm:ratpron}.
    \item We provide an intuitive algorithm to recover the parameters of Prony signals (see Eq.~\eqref{eq:PronySig}) using a rational variant of the generalized operator-based Prony algorithm~\cite{dozsa2025system}. In addition to showing that any Prony problem can be solved this way, we also provide experiments that indicate that numerical issues associated with solving a Vandermonde system to recover linear parameters of Prony signals can be bypassed through the use of appropriate rational orthogonal expansions. The novelties are given in Theorem~\ref{thm:lin} and subsection~\ref{sec:exp1}.
    \item For the first time, we show how GB schemes can be used to solve \textit{any} Prony problem. This is especially useful, when the number of basis functions ($M$ in Eq.~\eqref{eq:PronySig}) is unknown a priori, or when it is known to be large and a reduced-order model of $f$ is required. Bernoulli schemes in this case provide asymptotic guarantees that the identified nonlinear parameters match those of the original signal $f$. This result is stated in Theorem~\ref{thm:pronbern}.
    \item In section~\ref{sec:ex}, we provide several examples and numerical experiments which show the benefit of our proposed algorithms. 
\end{enumerate}

The remainder of this paper is structured as follows. In section~\ref{sec:genPro} we review the GOP method introduced in~\cite{peter2013generalized} and~\cite{stampfer2020generalized}. In addition, we discuss a rational variant of the GOP scheme~\cite{dozsa2025system}, that we call generalized rational operator-based Prony (GROP) method, and that can be used to identify SISO LTI systems in the frequency domain. In section~\ref{sec:genBer} we review generalized Bernoulli methods and their connection to rational Prony problems. In section~\ref{sec:genRat} we state our main findings and results. More precisely, we show that any general Prony problem can be written as a  rational one. In addition, we argue that certain aspects of $H_2$ spaces can be exploited to improve the numerical stability of the GROP algorithm. Finally, we introduce Bernoulli schemes which can be used to solve any Prony problem. Section~\ref{sec:ex} contains our examples and numerical experiments. In section~\ref{sec:conc} we draw our conclusions and describe future research directions.

\section{The generalized operator-based Prony method}
\label{sec:genPro}
We shall begin with a short review of the generalized operator-based Prony (GOP) method as introduced in~\cite{stampfer2020generalized}. Whenever possible we will follow the notations introduced in~\cite{peter2013generalized} and~\cite{stampfer2020generalized}. For an excellent and deeper discussion of the GOP method, we recommend~\cite{stampfer2019generalized}.

\subsection{Preliminaries}
\label{sec:prelims}
 First, we recite some important definitions and theorems. Consider the Prony problem defined in Eq.~\eqref{eq:PronySig} with all of the assumptions mentioned there. Without loss of generality, one may assume that there exists a linear operator $A : V \to V$, such that the nonlinear parameters $\lambda_k$ are included among the eigenvalues of $A$. Denote the \emph{active} part of the (point) spectrum of $A$ by $\Lambda := \{\lambda_1,\ldots,\lambda_M\}$ with $\lambda_k \neq \lambda_j$ for all $j \neq k$ and consider the set
\begin{equation}
     \label{eq:ma}
     \cM := \mathcal{M}(A,\Lambda) := \operatorname{span} \{v_{\lambda_k}:\ k=1,\ldots,M \} = \left\{ f = \sum_{k=1}^M c_k v_{\lambda_k}: \ c_k \in \mathbb{C}\right\}.
\end{equation}
This set is called \emph{signal space} which is indeed a vector space for any fixed $\Lambda$. However, we will typically assume that $c_k \neq 0$ $(k = 1,\ldots,M)$, because then the signal complexity coincides the number of Prony atoms.
In practical cases, not only is $A$ assumed to have a point spectrum, but we also assume that $f$ is sparse in the sense that at most $M$ pairwise distinct eigenvalues belonging to $\mathbb{C}$ will be considered.  For simplicity, we shall denote both the full and this restricted point spectrum of $A$ with $\Lambda(A)$ or simply $\Lambda$ henceforth and specify the difference whenever necessary. In fact, we shall consider
\begin{equation}
    \label{eq:lambset}
    \Lambda = \Lambda(A) := \{ \lambda_1, \lambda_2, \ldots, \lambda_M \} \subset \mathbb{C} \quad (M \in \mathbb{N})
\end{equation}
and use the notation $\lambda_k \ (k=1,\ldots,M)$ to refer to the nonlinear parameters of Prony signals. Following~\cite{stampfer2019generalized}, we shall say that the operator $A$ \textit{generates} the signal space $\cM$, or $A$ is a \emph{generator} of $\cM$. First we define of iteration operators as introduced in~\cite{stampfer2019generalized, stampfer2020generalized}:

\begin{definition}[Iteration operator]
    \label{def:iterop}
    Let $V$ be a vector space and $A : V \to V$ be a generator of the signal space $\cM$. Furthermore, let $\varphi: \Lambda(A) \to \mathbb{C}$ be injective. When a bounded linear operator
    $$
        \Psi_{\varphi} : \cM \to \cM
    $$
    satisfies 
    $$
        \Psi_{\varphi}v_{\lambda_k} = \varphi(\lambda_k) v_{\lambda_k} \quad (k=1,\ldots,M),
    $$
    it shall be referred to as an \emph{iteration operator}. 
\end{definition}
We note again, that in this paper, $\Lambda(A)$ denotes the restricted (sparse) spectrum of $A$, for a more general definition, we refer to~\cite[Def.~2.2.4]{stampfer2019generalized}. In order to recover the parameters of Prony signals using the GOP method, so-called evaluation schemes need to be defined~\cite{stampfer2019generalized, stampfer2020generalized}.

\begin{definition}[(Admissible) evaluation scheme] Let $V$ be a vector space and $A : V \to V$ be a generator of the signal space $\cM$.
A sequence $\mathcal{F} := {(\mathcal{F}_m)}_{m \ge 0}$ of linear functionals
$$
    \mathcal{F}_m : \cM \to \mathbb{C} \quad (m \in \mathbb{N} \cup \{0\})
$$
is called an \emph{evaluation scheme}. Furthermore,  define $E_{\mathcal{F}} := \left[ \mathcal{F}_m(v_{\lambda_k}) \right]_{m=0, k = 1}^{N,M} \in \C^{(N+1) \times M}$. The evaluation scheme is called \emph{admissible}, if for any $N \geq M - 1$ we have $\rank(E_\mathcal{F}) = M$.
\end{definition}

An important special case is given by $\mathcal{F}_m := \funF \circ \Psi_m$, where $\Psi_m : \cM \to \cM$ are appropriate linear operators and $\funF : \cM \to \mathbb{C}$ is a linear functional which satisfies $\funF(v_{\lambda_k}) \neq 0 \ (k=1,\ldots,M)$. Often, such an evaluation scheme is generated by powers of a single iteration operator $\Psi_\varphi$, i.e., $\mathcal{F}_m := \funF \circ \Psi_\varphi^m$. In this case, the evaluation scheme is referred to as a \emph{canonical evaluation scheme}. In~\cite{stampfer2019generalized}, Stampfer shows that a canonical evaluation scheme is always admissible. 
Using evaluation schemes, one can define so-called sampling schemes~\cite{stampfer2019generalized, stampfer2020generalized}.

\begin{definition}[(Realizable/admissible) sampling scheme]
\label{def:evalscheme}
    Consider an evaluation scheme ${(\mathcal{F}_m)}_{m \ge 0}$ and an iteration operator $\Psi_{\varphi} : \cM \to \cM$ acting on the signal space $\cM$. The family of maps ${(S_{m, j})}_{m,j \ge 0}$ defined by
    $$
        S_{m, j} := \mathcal{F}_{m} \circ \Psi_{\varphi}^{j} : \cM \to \mathbb{C} \quad (m, j \in \mathbb{N} \cup \{0\})
    $$
    is called a \emph{sampling scheme} for $\cM$. A sampling scheme is called \emph{admissible}, if ${(\mathcal{F}_m)}_{m \ge 0}$ is admissible. A sampling scheme is called \emph{realizable}, if it can be written as a family of linear functionals that, when applied to $f$ (the Prony signal from Eq.~\eqref{eq:PronySig}), can be evaluated without any  derivatives of $f$. 
\end{definition}
Fixing upper indices $N,M$ in the sampling scheme allows us to acquire the so-called sampling matrix $\mathcal{X}_{N,M} := \left[ S_{m, j}(f) \right]_{m, j = 0}^{N, M} \in \C^{(N+1) \times (M+1)}$. 

The key idea behind the GOP method is the following theorem proven by Stampfer and Plonka in~\cite{stampfer2020generalized}.

\begin{theorem}[Generalized operator-based Prony (GOP) method]
\label{thm:GOP}
    Let ${(S_{m, j})}_{m,j \ge 0} := {(\mathcal{F}_m \circ \Psi_{\varphi}^{j})}_{m,j \ge 0}$ be an admissible sampling scheme over the $M$-dimensional signal space $\cM$ as in \eqref{eq:ma}. Then every $f \in \cM$ is uniquely determined by the sampling matrix
    \begin{equation}
        \label{eq:sampmat}
        \mathcal{X}_{N,M-1} := \left[ \mathcal{F}_m(\Psi_{\varphi}^{j} f)\right]_{m,j=0}^{N,M-1},
    \end{equation}
    where $N \geq M-1$.
\end{theorem}

The proof of Theorem~\ref{thm:GOP} can be found in~\cite{stampfer2019generalized} and~\cite{stampfer2020generalized}. We now proceed to demonstrate the GOP algorithm. Consider the so-called Prony polynomial
\begin{equation}
    \label{eq:pronpol}
    P_{\Lambda(A)}^{\varphi}(z) := \prod_{k=1}^M (z - \varphi(\lambda_k)) = \sum_{k=0}^M p_k z^{k} \quad (z \in \mathbb{C}).
\end{equation}
With the conditions of Theorem~\ref{thm:GOP}, the so-called annihilation equation holds, i.e., 
\begin{equation}
    \label{eq:anniheq}
    \sum_{k=0}^{M} p_k \Psi_{\varphi}^k f = 0.
\end{equation}
Exploiting this, the monic nature of the polynomial defined in Eq.~\eqref{eq:pronpol} (i.e., $p_M = 1$), and the admissibility of the evaluation scheme ${(\mathcal{F}_m)}_{m \ge 0}$, we obtain the Hankel-like system (see~\cite[Lem.~2.3.8]{stampfer2019generalized})
\begin{equation}
    \label{eq:hank}
    \begin{bmatrix}
         \mathcal{F}_0(\Psi_{\varphi}^{0} f) & \mathcal{F}_0(\Psi_{\varphi}^{1} f) & \ldots & \mathcal{F}_0(\Psi_{\varphi}^{M-1} f) \\
         \mathcal{F}_1(\Psi_{\varphi}^{0} f) & \mathcal{F}_1(\Psi_{\varphi}^{1} f) & \ldots & \mathcal{F}_1(\Psi_{\varphi}^{M-1} f) \\
         \vdots & \vdots & \ddots & \vdots \\
         \mathcal{F}_N(\Psi_{\varphi}^{0} f) & \mathcal{F}_N(\Psi_{\varphi}^{1} f) & \ldots & \mathcal{F}_N(\Psi_{\varphi}^{M-1} f)
    \end{bmatrix} \begin{bmatrix}
         p_0  \\
         p_1 \\
         \vdots \\
         p_{M-1}
    \end{bmatrix} = -\begin{bmatrix}
         \mathcal{F}_0(\Psi_{\varphi}^{M} f) \\
         \mathcal{F}_1(\Psi_{\varphi}^{M} f) \\
         \vdots \\
         \mathcal{F}_N(\Psi_{\varphi}^{M} f)
    \end{bmatrix}.
\end{equation}
According to Theorem~\ref{thm:GOP}, this system has a unique solution and the zeros $\varphi(\lambda_k)$ for $k=1,\ldots,M$ of the polynomial $P_{\Lambda(A)}^{\varphi}$ can be recovered using a preferred numerical scheme. Since $\varphi$ is injective, the nonlinear parameters $\lambda_k$ can also be identified. The linear parameters can be recovered using~\cite[Lem.~2.2.14]{stampfer2019generalized}. This lemma states that if the conditions of Theorem~\ref{thm:GOP} are satisfied and the nonlinear parameters in $\Lambda(A)$ have been successfully recovered, then for any choice of $\omega := (\omega_0, \ldots, \omega_N)^\T \in \mathbb{C}^{N+1} \setminus \{0\}$ such that $\Gamma := \operatorname{diag} \left( \left[ \sum_{m=0}^N \omega_m \mathcal{F}_m(v_{\lambda_k})\right]_{k=1}^M \right)$ is invertible, the equation
\begin{equation}
    \label{eq:Vander}
    \mathcal{X}_{N, M-1}^{\T} \omega = \left[ \varphi^j(\lambda_k) \right]_{j=0, k=1}^{M-1, M}  {[\tilde{c}_{\lambda_k}]}_{k=1}^M,
\end{equation}
holds, where ${[\tilde{c}_{\lambda_k}]}_{k=1}^M := \Gamma {[c_{\lambda_k}]}_{k=1}^M$. This is a Vandermonde-like equation for the variable ${[\tilde{c}_{\lambda_k}]}_{k=1}^M$ which can be poorly conditioned, however methods exist to mitigate this problem (see, e.g.,~\cite{higham1988fast}). In this paper, we show how a rational take on the GOP method can be leveraged to introduce a triangular system to recover the linear parameters. Furthermore, we shall empirically verify that this triangular system is often better conditioned than the matrix appearing in Eq.~\eqref{eq:Vander}.  

\subsection{The GOP method for rational problems}
\label{sec:pronyrat}

In this section we shortly summarize the main results in~\cite{dozsa2025system} which will be important for the results presented in this paper. The motivation behind~\cite{dozsa2025system} is to develop a realization of the GOP method, which can be used to identify so-called discrete-time single-input single-output linear time-invariant (SISO LTI) dynamical systems.
These can be characterized by the discrete convolution equation
\begin{equation}
    \label{eq:sisoTime}
    y = h \ast u,
\end{equation}
where $u, y, h \in \ell$ are called the input, the output, and the impulse response of the system, respectively. It is usually assumed, that for a bounded input sequence $u$, the output $y$ is also bounded. This is known as bounded-input bounded-output (BIBO) stability~\cite{szidar, vandenhof} and it coincides with the condition $h \in \ell_1$, where $\ell_1$ denotes the set of absolutely summable complex sequences. Because we deal with unilateral sequences $h$, i.e., $h_n = 0$ for all $n < 0$, we also implicitly assume that the system is causal which means that future input values cannot influence the current output of the system.


Under these assumptions, it can be shown (see, e.g.,~\cite{vandenhof, dozsa2025system}) that for all $n \ge 0$, the impulse response defining the system behavior can be written as
\begin{equation}
    \label{eq:impresp}
    h_n = \sum_{k=1}^{M} c_k \overline{\lambda_k}^n,
\end{equation}
where $M \in \mathbb{N}$ is known as the minimal system order or McMillan degree. 
BIBO stability clearly implies that $|\lambda_k| < 1 \ (k=1,\ldots,M)$. For simplicity, we henceforth assume that $\lambda_k \neq \lambda_j$ if $k \neq j$.

Since according to Eq.~\eqref{eq:impresp}, the impulse response can be written as a weighted sum of complex powers, it might be tempting to apply the classical Prony method to recover the parameters $c_k$ and $\lambda_k$. This has indeed been done, for example in~\cite{foyen2018prony}. However, the application of the classical Prony scheme comes at the cost of numerical problems and noise intolerance. In order to remedy this, in~\cite{dozsa2025system} the authors propose using the GOP method with the following sampling scheme. 

First, define the $\mathcal{Z}$-transform of a unilateral complex sequence as
\begin{equation}
    \label{eq:ztrans}
    \mathcal{Z}[x](z) := \sum_{n=0}^{\infty} x_n z^n \quad (z \in \mathbb{C}), 
\end{equation}
whenever the sum exists. In~\cite{dozsa2025system}, the authors propose to consider the system in the frequency domain by applying the $\mathcal{Z}$-transform to both sides of Eq.~\eqref{eq:sisoTime}, i.e., 

\begin{equation}
    \label{eq:sisofreq}
    \mathcal{Z}[y](z) = \mathcal{Z}[h \ast u](z) \iff Y(z) = H(z) U(z),
\end{equation}
where $H(z)$ is referred to as the transfer function of the system. If the system is causal and BIBO stable with an impulse response satisfying Eq.~\eqref{eq:impresp}, then the transfer function $H(z) = \mathcal{Z}[h](z)$ is a rational function that belongs to the Hardy space 
\begin{equation*}
    H_{2}(\mathbb{D}) := \left\{ f \in \mathcal{A}(\mathbb{D}) : \ \sup_{r < 1} \left( \frac{1}{2 \pi} \int_{-\pi}^{\pi} \left| f(r \e^{\ri t}) \right|^2 \,\d t \right)^{1/2} < \infty \right\},
\end{equation*}
where $\mathcal{A}(\mathbb{D})$ denotes the space of complex analytic functions defined on the open unit disk. Note that $H_{2}(\mathbb{D})$ is a Hilbert space with the inner product
\begin{equation}\label{eq:H2def}
 {\langle f,g \rangle}_{H_{2}(\mathbb{D})} := \frac{1}{2 \pi} \int_{-\pi}^{\pi} f(\e^{\ri t})\overline{g(\e^{\ri t})} \,\d t.
\end{equation}
Indeed, since $h \in \ell_1$, the transfer function $H = \mathcal{Z}[h]$ belongs to $H_{\infty}(\mathbb{D})$, the set of bounded functions in $\mathcal{A}(\mathbb{D})$. The statement thus follows from $H_{\infty} (\mathbb{D})\subset H_2(\mathbb{D}) $. In fact, with the above assumptions the transfer function $H$ can be written as
\begin{equation}
   \label{eq:trfH}
   \mathcal{Z}[h](z) = H(z) := \sum_{k=1}^{M} c_k r_{\lambda_k}(z) = \sum_{k=1}^{M} c_k \frac{1}{1 - \overline{\lambda_k}z}.
\end{equation}

\begin{remarks}
\label{rem:mirrorpole}
We note that the definition of the transfer function $H$ is different from the definition usually found in the systems theory literature (see, e.g.,~\cite{vandenhof}). By Eq.~\eqref{eq:trfH}, the poles of $H$ are $\frac{1}{\overline{\lambda_k}}$, $k=1,\ldots, M$. They are the mirror image reflections of the nonlinear parameters $\lambda_k$ across $\mathbb{T}$. If the transfer function $H$ is real-rational, then these are the poles of the function $\hat{H}(z) := H(\frac{1}{z})$ for all $z \in \mathbb{C} \setminus \overline{\mathbb{D}}$ which is the transfer function typically found in the literature. However, this is different here, because we use a non-standard definition of the $\mathcal{Z}$-transform where the  exponents of $z$ are non-negative (and not non-positive as usual). Because of these properties, the nonlinear parameters $\lambda_k$ are also often referred to as poles of the system (or more precisely, poles of the transfer function $\hat{H}$ and ``mirror image poles" of the transfer function $H$).

The benefit of the formulation in~\eqref{eq:trfH} is that the nonlinear parameters $\lambda_k$ will be strictly inside $\mathbb{D}$ and therefore satisfy $|\lambda_k| < 1$. In addition, this formulation can also be utilized when investigating the convergence properties of GB methods using hyperbolic geometry~\cite{siamdozsa}.
\end{remarks}

The transfer function in~\eqref{eq:trfH} can be understood as a Prony signal (see Eq.~\eqref{eq:PronySig}), with the choices $V := H_2(\mathbb{D})$, $v_{\lambda_k} := r_{\lambda_k}$ and $\Lambda := \{\lambda_1, \ldots, \lambda_M\} \subset \mathbb{D}$. Since $\{r_{\lambda_1},\ldots,r_{\lambda_M}\}$ are linearly independent, the transfer function $H$ lies in an $M$-dimensional subspace of $H_2(\mathbb{D})$ given by
\begin{equation}
    \label{eq:modspace}
    \mathcal{H}_{\Lambda} := \operatorname{span} \left\{ r_{\lambda_k}: \ k=1,\ldots,M \right\}  \subset H_2(\mathbb{D}),
\end{equation}
where the $H_2(\mathbb{D})$-functions $r_{\lambda_k}$ are defined as
\begin{equation}\label{eq:def_rlambda}
r_{\lambda_k}(z) := \frac{1}{1-\overline{\lambda_k}z} \quad (z\in \mathbb{D}, \ k=1,\ldots,M).
\end{equation}
The subspace $\mathcal{H}_{\Lambda}$ is often referred to as a model space~\cite{bprodbook}. 

In~\cite{dozsa2025system}, the following two observations are exploited. First, the adjoint $H_2(\mathbb{D})$ shift operator defined by
\begin{equation}
    \label{eq:adjshift}
    (S^{\ast}f)(z) := \begin{cases}
     \frac{f(z) - f(0)}{z}, & \text{if } z \in \mathbb{D} \setminus\{0\} \\ f'(0), & \text{otherwise} \end{cases} \quad (f \in H_2(\mathbb{D})),
\end{equation}
satisfies
\begin{equation*}
    S^{\ast}r_{\lambda_k} = \overline{\lambda_k} r_{\lambda_k}  \quad (k=1,\ldots,M).
\end{equation*}
It is also not difficult to see that if $f \in H_2(\mathbb{D})$, then $S^{\ast}f \in H_2(\mathbb{D})$. Thus, the linear operator $S^{\ast}$ generates the Prony signal $H$ (and the signal space $\cM:=\mathcal{M}(S^{\ast},\Lambda$)). Using this fact, in~\cite{dozsa2025system} an appropriate evaluation scheme ${(\mathcal{F}_m)}_{m \ge 0}$ of the form 
\begin{equation}\label{eq:proposamp}
\mathcal{F}_m := \funF \circ \left( S^{\ast} \right)^m \quad (m \in \mathbb{N}\cup\{0\})
\end{equation}
for some linear functional $\funF : H_2(\mathbb{D}) \to \mathbb{C}$ is considered. 
In~\cite{dozsa2025system}, $\funF$ is chosen as
\begin{equation}\label{eq:defF}
     \funF f := \frac{1}{2 \pi} \int_{-\pi}^{\pi} f(\e^{\ri t})\, \d t \quad (f \in H_2(\mathbb{D})).
\end{equation}
We note that in practice, $f$ is usually only available in a discretely sampled form on $\mathbb{T}$. In this case, a numerical quadrature can be used to approximate $\funF$ from Eq.~\eqref{eq:defF}. 
We refer  to~\cite{cohen2017optimal} for some strategies on choosing the sampling points and for the derivation of error estimates.

If $H \in \mathcal{H}_{\Lambda}$ is assumed to be known on $\mathbb{T}$, then the sampling matrix (Eq.~\eqref{eq:sampmat}) contains entries of the form
\begin{equation*}
    \mathcal{F}_m H = \frac{1}{2 \pi} \sum_{k=1}^{M} c_k \int_{-\pi}^{\pi} \left(S^{*}\right)^m r_{\lambda_k}(\e^{\ri t}) \,\d t = \frac{1}{2 \pi} \sum_{k=1}^{M} c_k \overline{\lambda}_k^m \quad (m \in \mathbb{N} \cup \{0\}).
\end{equation*}

In~\cite{dozsa2025system}, the evaluation scheme~\eqref{eq:proposamp} is shown to increase the noise tolerance of the method through numerical experiments. 
The sampling matrix for $N = M-1$ generated by the evaluation scheme in Eq.~\eqref{eq:proposamp} and the corresponding Hankel system take the form
\begin{equation}\label{eq:hankRat}
    \begin{bmatrix}
         \funF\big(( S^{\ast})^{0} H\big) & \funF\big(( S^{\ast} )^1 H\big) & \ldots & \funF\big( ( S^{\ast} )^{M-1} H\big) \\
         \funF\big(( S^{\ast})^{1} H\big) & \funF\big(( S^{\ast})^{2} H\big) & \ldots & \funF\big(( S^{\ast})^{M} H\big) \\
         \vdots & \vdots & \ddots & \vdots \\
         \funF\big(( S^{\ast} )^{M-1} H\big) & \funF\big(( S^{\ast})^{M} H\big) & \ldots & \funF\big(( S^{\ast})^{2M-2} H\big)
    \end{bmatrix} \begin{bmatrix}
         p_0  \\
         p_1 \\
         \vdots \\
         p_{M-1}
    \end{bmatrix} = - \begin{bmatrix}
         \funF\big(( S^{\ast} )^{M} H\big) \\
         \funF\big(( S^{\ast} )^{M+1} H\big) \\
         \vdots \\
         \funF\big(( S^{\ast})^{2M-1} H\big)
    \end{bmatrix},
\end{equation}
where $\funF$ is defined according to Eq.~\eqref{eq:proposamp}. Using the coefficients $p_0, p_1, \ldots, p_{M-1}$, it is possible to recover the zeros of the polynomial
\begin{equation}
    \label{eq:ratPronPol}
    P_{\Lambda(S^{\ast})}(z) := \prod_{k=1}^{M} (z - \overline{\lambda}_k) = \sum_{k=0}^{M} p_k z^k.
\end{equation}
Once the nonlinear parameters $\lambda_k \ (k=1,\ldots,M)$ are found, it is possible (see~\cite{dozsa2025system}) to find the residues $c_k$ by solving
\begin{equation}
    \label{eq:residVander}
   \begin{bmatrix}
         1 & 1 & \ldots & 1 \\
         \overline{\lambda}_1 & \overline{\lambda}_2 & \ldots & \overline{\lambda}_{M} \\
         \vdots & \vdots & \ddots & \vdots \\
         \overline{\lambda}_1^{M-1} & \overline{\lambda}_2^{M-1} & \ldots & \overline{\lambda}_{M}^{M-1} \\
    \end{bmatrix}\begin{bmatrix}
         \tilde{c}_1  \\
         \tilde{c}_2 \\
         \vdots \\
         \tilde{c}_M
    \end{bmatrix}  = \begin{bmatrix}
         \funF\big(( S^{\ast} )^{0} H\big) \\
         \funF\big(( S^{\ast})^{1} H\big) \\
         \vdots \\
         \funF\big(( S^{\ast})^{M-1} H\big)
    \end{bmatrix},
\end{equation}
where $\tilde{c}_k := c_k \funF(r_{\lambda_k}) \ (k=1,\ldots,M)$. Note that  Eq.~\eqref{eq:residVander} can be obtained as a special case of Eq.~\eqref{eq:Vander} with the proposed sampling scheme. 

In section~\ref{sec:genRat} we are going to show the following claims: First, even though the rational sampling scheme presented in this section is indeed a special case of the GOP framework, in fact, \textit{any} Prony problem (recovering the linear and nonlinear parameters of a Prony signal using the GOP method) can be solved using it. The rational framework proposed here has the following important benefit. In this framework (as shown in section~\ref{sec:genRat}) it is possible to easily replace Eq.~\eqref{eq:residVander} with a triangular system of equations to recover the linear parameters. 

\section{Generalized Bernoulli schemes}
\label{sec:genBer}

Next we are going to review GB schemes which can be used to recover the (dominant) nonlinear parameters $\lambda_k$ of rational functions in $\mathcal{H}_{\Lambda}$ in an iterative fashion, even if $M$ is unknown. Once we have clarified that indeed all general Prony problems can be reduced to the rational case introduced in subsection~\ref{sec:pronyrat}, we are going to be able to apply Bernoulli algorithms to recover the nonlinear parameters of general Prony signals. The linear parameters of the signals can then be recovered by solving a triangular equation system proposed in section~\ref{sec:genRat}. This is especially useful if $M$ is unknown (but assumed to be finite), or when $M$ is large and instead of the Prony signal we are interested in identifying a reduced-order model, whose dominant parameters match those of the original signal. Our discussion will focus on the findings of~\cite{siamdozsa}, with brief mentions of previous and further generalizations. 

Let $M \in \mathbb{N}$ and consider $\Lambda := \{ \lambda_1, \lambda_2, \ldots, \lambda_M\} \subset \mathbb{D}$. Similarly to subsection~\ref{sec:pronyrat}, we shall focus on $M$-dimensional subspaces of $H_2(\mathbb{D})$ given by Eq.~\eqref{eq:modspace}. In subsection~\ref{sec:pronyrat} we have already seen that transfer functions of stable SISO LTI systems belong to such model spaces. 

Henceforth assume that we are interested in recovering the nonlinear parameters of $H \in \mathcal{H}_{\Lambda}$ (i.e., finding fully or partially the set $\Lambda$) given access to a (discrete sampling) of $H_{|\mathbb{T}}$. Let $\gamma : \mathbb{D} \to \mathbb{D}$ be arbitrary. We shall call the nonlinear parameter $\lambda_1 \in \Lambda$ dominant in $\Lambda$ with respect to $\gamma$, if
\begin{equation}
    \label{eq:gdpole}
    |\gamma(\lambda_1)| > |\gamma(\lambda_k)| \quad (k=2,\ldots,M).
\end{equation}
Similarly to the introduction, here  the discussed results hold, if
the $\gamma$-dominant nonlinear parameter $\lambda_1$ is unique as in 
Eq.~\eqref{eq:gdpole}.
If $\gamma(z) := z \ (z \in \mathbb{D})$ and $\lambda_1$ satisfies Eq.~\eqref{eq:gdpole}, then (as we have seen in Eq.~\eqref{eq:dpole}), we call $\lambda_1$ dominant. 

If a dominant nonlinear parameter exists, then Bernoulli's classical method (see Eq.~\eqref{eq:BernOrig}) can be used to recover it. Unfortunately, the existence of a unique dominant nonlinear parameter can seldom be guaranteed a priori. In order to overcome this, consider the so-called Takenaka-Malmquist (TM) function system~\cite{takenakaorig, malmquistorig}
\begin{equation}
    \label{eq:MT}
    \Phi_n^{\boldsymbol{a}}(z) := \frac{\sqrt{1 - |a_n|^2}}{1 - \overline{a}_n z} \cdot \prod_{j=0}^{n-1} B_{a_j}(z) \quad (\boldsymbol{a} \in \ell_{\mathbb{D}}, \ z \in \overline{\mathbb{D}}, \ n \in \mathbb{N} \cup \{0\}),
\end{equation}
where $\ell_{\mathbb{D}}$ denotes the set of complex sequences in $\mathbb{D}$ and
\begin{equation}
    \label{eq:Blaschfac}
    B_a(z) := \frac{z - a}{1 - \overline{a}z} \quad (a \in \mathbb{D}, \ z \in \overline{\mathbb{D}})
\end{equation}
is a so-called Blaschke factor. Both Blaschke factors and TM functions have interesting and important properties with far reaching consequences for Hardy spaces. For example, Blaschke factors are bijective self maps on $\mathbb{T}$ and $\mathbb{D}$ and can be used to describe congruence transformations in the Poincar\'e disc model of the Bolyai-Lobachevskiy hyperbolic geometry. In addition, they form a group with respect to function composition. Products of Blaschke factors are called Blaschke products and are also self maps on $\mathbb{T}$ and $\mathbb{D}$. Blaschke products are inner functions and play an important role in the factorization of Hardy spaces. They can also be interpreted as transfer functions of all pass filters~\cite{vandenhof, AllPassFilters}. 

TM functions~\eqref{eq:MT} form an orthonormal system in $H_2(\mathbb{D})$. They are also complete in $H_2(\mathbb{D})$ provided that the generating sequence $\boldsymbol{a}$ satisfies the so-called Sz\'asz condition~\cite{vandenhof, bprodbook}, i.e., 
\begin{equation}
    \label{eq:szasz}
    \sum_{n=0}^{\infty} (1 - |a_n|) = \infty.
\end{equation}
By \cite[Prop.~12.3.4 (b)]{bprodbook}, for any $f \in H_2(\mathbb{D})$ we have
\begin{equation}
\label{eq:cauchyf}
    {\langle r_{\lambda},f \rangle}_{H_2(\mathbb{D})} = \frac{1}{2 \pi} \int_{-\pi}^{\pi} \frac{\overline{f(\e^{\ri t})}}{1 - \overline{\lambda} \e^{\ri t}}\, \d t = \overline{f(\lambda)}.
\end{equation}
Thus, if $H \in \mathcal{H}_{\Lambda}$, its Fourier coefficients with respect to a TM function can be written as
\begin{equation}
    \label{eq:TMcoeffs}
    c^{\boldsymbol{a}}_n := \langle H, \Phi_n^{\boldsymbol{a}} \rangle_{H_2(\mathbb{D})} = \sum_{k=1}^M c_k \langle r_{\lambda_k}, \Phi_n^{\boldsymbol{a}} \rangle_{H_2(\mathbb{D})} =  \sum_{k=1}^{M} c_k \overline{\Phi_n^{\boldsymbol{a}}(\lambda_k)}.
\end{equation}
We call a TM system $p$-periodic for some $p \in \mathbb{N}$, if for all $n \in \{0,\ldots,p-1\}$ the entries of the generating sequence $\boldsymbol{a}$ satisfy
\begin{equation*}
    a_n := a_{n + kp} \quad (k \in \mathbb{N}).
\end{equation*}
If a generating sequence $\boldsymbol{a}$ is $p$-periodic, using an abuse of notation we shall denote by $\boldsymbol{a}$ a single period $\boldsymbol{a} := (a_0, a_1, \ldots, a_{p-1}) \in \mathbb{D}^p$. Based on Eq.~\eqref{eq:MT}, it is then not difficult to see that $p$-periodic TM functions satisfy
\begin{equation*}
    \Phi_{n + kp}^{\boldsymbol{a}} = \Phi_n^{\boldsymbol{a}} B_{\boldsymbol{a}}^k \quad (k \in \mathbb{N}),
\end{equation*}
where
\begin{equation*}
    B_{\boldsymbol{a}}(z) := \prod_{j=0}^{p-1} B_{a_j}(z) = \prod_{j=0}^{p-1} \frac{z - a_j}{1 - \overline{a_j} z}.
\end{equation*}
For ease of notation, for $p$-periodic TM systems, we use the notation
\begin{equation}
    \label{eq:index}
    \nu_k^{n, p} := n + pk \quad (k \in \mathbb{N} \cup \{0\}, \ 0 < n < p, \ p \geq 1).
\end{equation}

We note that a $p=1$ periodic TM system coincides with the so-called discrete Laguerre functions. Furthermore, if $\boldsymbol{a}$ is chosen as the constant $0$ sequence, the generated TM system coincides with the trigonometric functions. Finally, it should be noted that for any $p \in \mathbb{N}$, condition~\eqref{eq:szasz} is satisfied and the generated TM system is complete in $H_2(\mathbb{D})$. 

The key idea behind GB schemes is to replace the trigonometric Fourier coefficients in Eq.~\eqref{eq:BernOrig} with appropriately chosen $p$-periodic TM-Fourier coefficients. This ensures that the method converges, even if a dominant pole (in the sense of Eq.~\eqref{eq:dpole}) does not exist. This result, obtained in~\cite{siamdozsa}, is summarized by the next theorem.

\begin{theorem}[Generalized Bernoulli (GB) schemes]
\label{thm:gbern}
    Let $\boldsymbol{a} \in \mathbb{D}^p$ for $p \geq 1$, and denote by $\{ \Phi_n^{\boldsymbol{a}} \}_{n=0}^{\infty}$ the corresponding $p$-periodic TM function system (see Eq.~\eqref{eq:MT}). Suppose that $H \in \mathcal{H}_{\Lambda}$ for $\Lambda = \{\lambda_1, \ldots, \lambda_M \} \subset \mathbb{D}$, where $\lambda_j \neq \lambda_k$ if $j \neq k$. Let $\gamma := B_{\boldsymbol{a}} = \prod_{j=0}^{p-1} \frac{z - a_j}{1 - \overline{a_j} z}$, and assume $\lambda_1 \in \Lambda$ is a $\gamma$-dominant nonlinear parameter of $H$ (see Eq.~\eqref{eq:gdpole}). Then,
    \begin{equation}
        \label{eq:genbern}
        \lim_{k \to \infty} \frac{\langle H, \Phi_{\nu_k^{n, p}+1}^{\boldsymbol{a}} \rangle_{H_2(\mathbb{D})}}{\langle H, \Phi_{\nu_k^{n, p}}^{\boldsymbol{a}} \rangle_{H_2(\mathbb{D})}} = \frac{\overline{\Phi_{n+1}^{\boldsymbol{a}}}(\lambda_1)}{\overline{\Phi_{n}^{\boldsymbol{a}}}(\lambda_1)}.
    \end{equation}
    Furthermore, the rate of convergence in Eq.~\eqref{eq:genbern} is $\mathcal{O}(\beta^k)$, where $\beta := \max_{k=2,\ldots,M} |\lambda_k| / |\lambda_1|$.
\end{theorem}

Below we list some important remarks regarding Theorem~\ref{thm:gbern}. 

\begin{remarks} 
\begin{enumerate}[label = \alph*)]
    \item The $\gamma$-dominant nonlinear parameter $\lambda_1$ can be recovered from the limit in Eq.~\eqref{eq:genbern}, since $\Phi_{n+1}^{\boldsymbol{a}} / \Phi_{n}^{\boldsymbol{a}}$ is invertible on $\mathbb{D}$~\cite{siamdozsa}. Indeed, due to Eq.~\eqref{eq:MT}, one has
$$
    U(z) := \frac{\Phi_{n+1}^{\boldsymbol{a}}(z)}{\Phi_{n}^{\boldsymbol{a}}(z)} = \kappa \frac{z - a_{n}}{1 - \overline{a_{n+1}} z},
$$
where
$$
    \kappa = \sqrt{(1 - |a_{n+1}|^2)/(1 - |a_n|^2)}.
$$
From this, we obtain $z = \frac{U(z)/ \kappa + a_n}{1 + \overline{a_{n+1}}U(z) / \kappa}$.

  \item  The theorem for the case $p=1$ (when $\Phi_n^{\boldsymbol{a}}, n \in \mathbb{N} \cup \{0\}$ is a discrete Laguerre system) was first proven in~\cite{SS1}.

    \item According to Theorem~\ref{thm:gbern}, it is possible to recover a $\gamma$-dominant nonlinear parameter of $H$, even without access to the order $M$.

    \item A $\gamma$-dominant nonlinear parameter will always exist, unless the components of $\boldsymbol{a} \neq 0$ are chosen from a specific set with zero measure from $\mathbb{D}$~\cite{SS1, SBS, siamdozsa}.

    \item The classical Bernoulli scheme (see Eq.~\eqref{eq:BernOrig}) can be applied using discrete Fourier coefficients. The generalized scheme from Theorem~\ref{thm:gbern} retains this property, if $H$ is sampled over a special grid (determined by $\boldsymbol{a}$) on $\mathbb{T}$. For a deeper discussion, we refer to~\cite[Sect.~4]{siamdozsa}. 

    \item Further generalizations of Theorem~\ref{thm:gbern} are possible. In particular, we refer to~\cite{med24}, where it is shown that non-orthogonal product systems from sub-algebras of $H_2(\mathbb{D})$ may be used instead of the TM system. We note that these more general constructions may also be used to recover the parameters of (irrational) general Prony signals in a similar fashion as it is discussed in section~\ref{sec:genRat}.
\end{enumerate}
\end{remarks}

Although Theorem~\ref{thm:gbern} can be used to recover a single $\gamma$-dominant nonlinear parameter of $H \in \mathcal{H}_{\Lambda}$, in practical applications one is interested in recovering all (or at least $M' < M$) dominant nonlinear parameters of the signal. This is possible by repeated application of Theorem~\ref{thm:gbern}. In~\cite{SBS, SSSAB} for example, it is shown that one can project $H$ onto the subspace $\mathcal{H}_{\Lambda \setminus \{\lambda_1\}}$
once $\lambda_1$ has been recovered. The projection operator can be fully expressed using $H_{|\mathbb{T}}$ and the already found nonlinear parameters. One can then apply Theorem~\ref{thm:gbern} again onto the projected signal. Another way to recover multiple nonlinear parameters is presented in~\cite{siamdozsa, med24}. Suppose the $\gamma$-dominant nonlinear parameter $\lambda_1$ has already been found by applying Theorem~\ref{thm:gbern} to $H$ with a TM system generated by some $\boldsymbol{a} \in \mathbb{D}^p$. Consider now the generating sequence $\boldsymbol{b} := (\boldsymbol{a}, \lambda_1) \in \mathbb{D}^{p+1}$. Applying Theorem~\ref{thm:gbern} again using the TM system corresponding to $\boldsymbol{b}$ ensures that $\lambda_1$ cannot be found again. Indeed, by Eq.~\eqref{eq:dpole} and Eq.~\eqref{eq:Blaschfac}, $\lambda_1$ cannot be $\gamma$-dominant for $\gamma := B_{\boldsymbol{b}}$. 

Finally, we note that we shall not discuss numerical problems related to the application of the GB scheme here. For example, for strategies to select the TM generating sequence $\boldsymbol{a}$, or how to approximate the limit in Eq.~\eqref{eq:genbern}, we refer to~\cite{siamdozsa} and~\cite{med24}. In the experiments detailed in section~\ref{sec:ex}, we also employ the numerical considerations suggested in these papers.

\section{Rational Prony and Bernoulli methods}
\label{sec:genRat}

In this section we state our main results related to generalized operator-based Prony and Bernoulli schemes. We begin by observing that any general Prony problem as introduced in sections~\ref{sec:intro} and~\ref{sec:genPro} has an equivalent pole finding problem associated with it. In the following theorem, we shall make use of the notion of the \emph{weighted} $\mathcal{Z}$-transform defined as
\begin{equation}
    \label{eq:zw}
    \mathcal{Z}_{w} [x](z) := \sum_{n=0}^{\infty} x_n  \left( \frac{z}{w} \right)^n \quad (z \in \mathbb{C}, \ w \in \mathbb{C} \setminus\{0\}, \ x \in \ell), 
\end{equation}
where we assume that the series on the right-hand side is convergent.
\begin{theorem}[Generalized operator-based Prony methods and pole finding problems]
\label{thm:ratpron}
    Let the signal space $\cM$ be defined as  in~\eqref{eq:ma} and let $f \in \cM$, i.e.,
    \begin{equation*}
        f = \sum_{k=1}^{M} c_k v_{\lambda_k}.
    \end{equation*}
   Then the following statements are satisfied:
    \begin{enumerate}[label = \alph*)]
    \item There exists a linear operator $\mathcal{G} : \cM \to H_2(\mathbb{D})$ such that
    \begin{equation}
        \label{eq:G}
        \mathcal{G}f := \sum_{k=1}^{M} \mu(c_k) r_{\varrho(\lambda_k)}, 
    \end{equation}
    where $\mu : \mathbb{C} \to \mathbb{C}$ and $\varrho: \mathbb{C} \to \mathbb{D}$ are injective functions and $r_{\varrho(\lambda_k)} \in H_2(\mathbb{D})$ with  $r_{\varrho(\lambda_k)}(z) := \frac{1}{1-\overline{\varrho(\lambda_k)}z}$ for all $k=1,\ldots,M$ (cf. Eq.~\eqref{eq:def_rlambda}).
    \item Assume that we have given a canonical evaluation scheme $\mathcal{F} = {(\mathcal{F}_m)}_{m \ge 0} = {(\funF \circ \Psi_\varphi^m)}_{m \ge 0}$ with an iteration operator $\Psi_\varphi:\cM \to\cM$. Then, $\mathcal{G}$ can be expressed explicitly as
    $$
        \mathcal{G}f := \mathcal{Z}_w[\mathcal{F}f],
    $$
    for a sufficiently large $w > 0$ and where $\mathcal{F}f:=(\mathcal{F}_0f,\mathcal{F}_1f,\ldots)$.
    \end{enumerate}
\end{theorem}

\begin{proof}
First we show the existence of such an operator, then we define it in a constructive manner without assuming a priori knowledge of the parameters $\lambda_k$ and $c_k$.
\begin{enumerate}[label = \alph*)]
\item  A linear mapping $\mathcal{G} : \cM \to H_2(\mathbb{D})$ which satisfies Eq.~\eqref{eq:G} exists. Indeed, since $\Lambda$ is bounded, there exists a $C_{\Lambda} > 0$ such that
\begin{equation}
    \label{eq:cf}
    \max_{k=1,\ldots,M}|\lambda_k| < C_{\Lambda}.
\end{equation}
Then, we can define $\mathcal{G} : \cM \to H_2(\mathbb{D})$ as
\begin{equation*}
    (\mathcal{G}f)(z) := \sum_{k=1}^{M} c_k r_{\lambda_k / C_{\Lambda}}(z) = \sum_{k=1}^{M} \frac{c_k}{1 - \overline{\lambda_k/ C_{\Lambda}}z} =: \hat{G}(z).
\end{equation*}
Then, $\hat{G}$ clearly belongs to an $M$-dimensional subspace of $H_2(\mathbb{D})$ and the linearity of $\mathcal{G}$ is obvious. 

\item Next, we define $\mathcal{G}$ in a constructive manner. 
Consider the sequence $g^f := \big(g_m^f\big)_{m \ge 0}$ with 
\begin{equation*}
    g^{f}_m := \mathcal{F}_m(f) = \sum_{k=1}^M c_k \funF( \Psi_\varphi^m v_{\lambda_k} ) = \sum_{k=1}^M c_k \varphi(\lambda_k)^m \funF( v_{\lambda_k}) = \sum_{k=1}^{M} \mu(c_k) \varphi(\lambda_k)^m\quad (m \in \mathbb{N} \cup \{0\}),
\end{equation*}
where $\mu(c_k) := c_k \funF(v_{\lambda_k})$. 

Choosing $w > \max_{k=1,\ldots,M} |\varphi(\lambda_k)|$, for $z \in \overline{\mathbb{D}}$ we get
\begin{align*}
    \label{eq:h2form}
         G(z) &= \mathcal{Z}_{w} \big[g^f\big] (z) = \sum_{n=0}^{\infty} g^f_n \left( \frac{z}{w} \right)^n = \sum_{k=1}^{M} \mu(c_k) \sum_{n=0}^{\infty} \left( \frac{\varphi(\lambda_k)}{w} \right)^n z^n =\sum_{k=1}^{M} \frac{\mu(c_k)}{1 - {\frac{\varphi(\lambda_k)}{w}}z},
\end{align*}
where in the notation of the theorem, we have $\varrho(\lambda) = \overline{\frac{\varphi(\lambda)}{w}}$ for all $\lambda \in \mathbb{C}$.
The function $G$ clearly belongs to $H_2(\mathbb{D})$. Using the above, the explicit form of the transformation $\mathcal{G}$ can be written as
\begin{equation*}
 \mathcal{G}f = G = \mathcal{Z}_w[\mathcal{F}f].
\end{equation*}
\end{enumerate}

\end{proof}

\begin{remarks}
Given a Prony problem as defined in Theorem~\ref{thm:ratpron} and an appropriate evaluation scheme, it is possible to recover the nonlinear parameters $\lambda_k$ and the linear parameters $c_k$ using the rational Prony approach discussed in subsection~\ref{sec:pronyrat}. 
\end{remarks}

In Theorem~\ref{thm:ratpron}, we require a canonical evaluation scheme to be known for the problem. From the proof it is clear, that the evaluation scheme is only required to map $f \in \cM$ to $g^f \in \ell$ in a way such that each component of $g^f$ can be written as a weighted sum of complex exponential terms similarly to the impulse response sequence of a BIBO stable and causal SISO LTI system. The following lemma relaxes the need to know a canonical evaluation scheme for the problem. In section~\ref{sec:ex} we provide numerical examples, where the sampling scheme is chosen according to this lemma.

\begin{lemma}[Exponential dual evaluation schemes]
\label{thm:samp}
Let $V$ be a vector space and $A : V \to V$ be a generator with active spectrum $\Lambda := \{ \lambda_1, \ldots, \lambda_M \} \subset \mathbb{C}$ of the signal space $\cM := \cM(A,\Lambda)$ as defined in~\eqref{eq:ma}.  Let $f \in \cM$, i.e.,
\begin{equation*}
    f = \sum_{k=1}^{M} c_k v_{\lambda_k}.
\end{equation*}
Let $\ell_{\cM^{\ast}}$ denote the vector space of sequences in the dual space $\cM^{\ast}$. Any sequence $\mathcal{F} := (\mathcal{F}_0, \mathcal{F}_1, \mathcal{F}_2, \ldots) \in \ell_{\cM^{\ast}}$ which satisfies
\begin{equation}
    \mathcal{F}_m f := \sum_{k=1}^{M} c_k\varphi(\lambda_k)^{m}  \quad (c_k \in \mathbb{C}, \ m \in \mathbb{N} \cup \{0\}),
\end{equation}
where $\varphi : \Lambda \to \mathbb{C}$ is injective, is an admissible evaluation scheme.
\end{lemma}

\begin{proof}
    It is clear that any sampling scheme based on $\mathcal{F}$ is realizable, because it does not depend on any information other than $f$. We recall that the condition for admissibility is that the matrix
    $$
        E_\mathcal{F} := [\mathcal{F}_m(v_{\lambda_k})]_{m=0, k=1}^{N, M} 
    $$
    has rank $M$ provided $N \geq M-1$. Writing $E_{\mathcal{F}}$ for $N = M-1$, we have
    \begin{align*}
        E_{\mathcal{F}} &= \begin{bmatrix}
        c_1 & c_2 & \ldots & c_M \\
         c_1  \varphi(\lambda_1) & c_2   \varphi(\lambda_2) & \ldots & c_M   \varphi(\lambda_M) \\
        \vdots & \vdots & \ddots & \vdots \\
        c_1   \varphi(\lambda_1)^{M-1}  & c_2  \varphi(\lambda_2)^{M-1} & \ldots & c_M   \varphi(\lambda_M)^{M-1} \\
        \end{bmatrix} \\ &=
        \begin{bmatrix}
        1 & 1 & \ldots & 1 \\
         \varphi(\lambda_1)  &  \varphi(\lambda_2) & \ldots &  \varphi(\lambda_M) \\
        \vdots & \vdots & \ddots & \vdots \\
         \varphi(\lambda_1)^{M-1}  &  \varphi(\lambda_2)^{M-1} & \ldots &  \varphi(\lambda_M)^{M-1} \\
        \end{bmatrix}
        \begin{bmatrix}
        c_1 & 0 & \ldots & 0 \\
        0  & c_2 & \ldots & 0 \\
        \vdots & \vdots & \ddots & \vdots \\
        0  & 0 & \ldots & c_M \\
        \end{bmatrix}.
    \end{align*}
    We observe that  $E_{\mathcal{F}}$ can be written as the product of a transposed Vandermonde matrix generated by a vector with pairwise different components and a non-singular diagonal matrix. Thus, $E_{\mathcal{F}}$ is regular with rank $M$. Increasing $N$ does not change this property.
\end{proof}

\begin{remarks}
    In~\cite{stampfer2019generalized}, Stampfer already proposes a so-called dual sampling scheme which is based on a similar evaluation scheme as the one proposed in Lemma~\ref{thm:samp} (see, e.g.,~\cite[Sect.~4.1]{stampfer2019generalized}). It should be emphasized that this sampling scheme assumes knowledge of an iteration operator $\Psi_{\varphi} : \cM \to \cM$. Lemma~\ref{thm:samp} makes no such assumptions which can be helpful in some applications (since sometimes obtaining iteration operators for a given Prony problem is difficult).
\end{remarks}

Using Theorem~\ref{thm:ratpron} and Lemma~\ref{thm:samp} it is (in theory) possible to obtain an equivalent SISO LTI identification task for any general Prony problem. One might ask why treating Prony problems this way is beneficial. One benefit of treating general Prony problems as pole finding tasks is that in this setting, the recovery of the linear parameters no longer requires a Vandermonde system (see Eq.~\eqref{eq:Vander} and~\eqref{eq:residVander}) which is usually poorly conditioned. Indeed, the following theorem holds.

\begin{theorem}[Triangular system for linear parameter recovery]
\label{thm:lin}
    Let $\Lambda := \{\lambda_1, \ldots, \lambda_M \} \subset \mathbb{D}$ be known for some $M \in \mathbb{N}$. Consider the function
    $$
        H := \sum_{k=1}^{M} c_kr_{\lambda_k} \in \mathcal{H}_{\Lambda}.
    $$
    Then, $[c_1, \ldots, c_M]^\T \in \mathbb{C}^M$ can be obtained by solving the upper triangular linear system of equations
\begin{equation}
        \label{eq:lintrig}
        \begin{bmatrix}
            \overline{\Phi_0^{\boldsymbol{a}}}(\lambda_1) & \overline{\Phi_0^{\boldsymbol{a}}}(\lambda_2) & \ldots & \overline{\Phi_0^{\boldsymbol{a}}}(\lambda_M) \\
            0 & \overline{\Phi_1^{\boldsymbol{a}}}(\lambda_2) & \ldots & \overline{\Phi_1^{\boldsymbol{a}}}(\lambda_M) \\
            \vdots & \vdots & \ddots & \vdots \\
            0 & 0 & \ldots & \overline{\Phi_{M-1}^{\boldsymbol{a}}}(\lambda_M)
        \end{bmatrix}  \begin{bmatrix}
             c_1  \\
             c_2 \\
             \vdots \\
             c_M
        \end{bmatrix}  = \begin{bmatrix}
             \langle H, \Phi_0^{\boldsymbol{a}} \rangle_{H_2(\mathbb{D})}  \\
             \langle H, \Phi_1^{\boldsymbol{a}} \rangle_{H_2(\mathbb{D})} \\
             \vdots \\
             \langle H, \Phi_{M-1}^{\boldsymbol{a}} \rangle_{H_2(\mathbb{D})}
        \end{bmatrix},
    \end{equation}
   where $\boldsymbol{a} := (\lambda_1, \ldots, \lambda_M) \in \mathbb{D}^M$ and $\{ \Phi_n^{\boldsymbol{a}} \}_{n=0}^{M-1}$ are the corresponding first $M$ Takenaka-Malmquist functions as in Eq.~\eqref{eq:MT}.
\end{theorem}

\begin{proof}
  It is well known, that ${\{ \Phi_n^{\boldsymbol{a}} \}}_{n=0}^{M-1}$ spans the model space $\mathcal{H}_{\Lambda}$. Exploiting this and the orthogonality of TM functions, we have
    \begin{equation*}
        H := \sum_{n=0}^{M-1} \langle H, \Phi_n^{\boldsymbol{a}} \rangle_{H_2(\mathbb{D})} \Phi_n^{\boldsymbol{a}}.
    \end{equation*}
    Examining the TM-Fourier coefficients as in Eq.~\eqref{eq:TMcoeffs}, for $n=0,\ldots,M-1$ we find
    \begin{equation*}
        \begin{split}
            & \langle H, \Phi_n^{\boldsymbol{a}} \rangle_{H_2(\mathbb{D})} = 
            \sum_{k=1}^M c_k \overline{\Phi_n^{\boldsymbol{a}}}(\lambda_k).
        \end{split}
    \end{equation*}
    By the definition of the TM functions (see Eq.~\eqref{eq:MT}) and the fact that $\boldsymbol{a} = (\lambda_1, \ldots, \lambda_M) \in \mathbb{D}^M$, we have
    $$
        \overline{\Phi_n^{\boldsymbol{a}}}(\lambda_k) = 0
    $$
    whenever $n > k-1$. Thus, we obtain Eq.~\eqref{eq:lintrig}.
    
\end{proof}

\begin{remarks}
    In practical applications, the $H_2(\mathbb{D})$ inner product cannot be evaluated. We note that in this case, discrete TM-Fourier coefficients (see, e.g.,~\cite[Sect.~4]{siamdozsa}) can also be used in Eq.~\eqref{eq:lintrig}.
\end{remarks}

Another important benefit of transforming general Prony problems to an $H_2$ setting is that GB algorithms can now be used to recover the nonlinear parameters of Prony signals. This is a direct consequence of Theorem~\ref{thm:ratpron}. As mentioned earlier, the use of GB schemes to recover nonlinear parameters $\lambda_k$ is most beneficial, when $M$ is unknown, since Bernoulli methods are capable of iteratively finding new dominant nonlinear parameters without a priori information about the number $M$ of Prony atoms. We note that for the Bernoulli method, the Prony atoms coincide with the rational functions given in Eq.~\eqref{eq:def_rlambda}). Another scenario, when Bernoulli methods might be beneficial is when $M$ is large and we are only interested in recovering the first $M' \ll M$ dominant nonlinear parameters of the function. Similar ideas, such as the dominant pole algorithm~\cite{martins2002dominant, rommes2008convergence} have long been used in model order reduction, however these assume an a priori known system model. In contrast, the Bernoulli scheme only assumes that it has access to (a method for sampling) the rational function to be identified on the unit circle. In this way, Bernoulli schemes can be viewed as data-driven dominant pole algorithms. 

The next theorem summarizes how the GB  algorithm can be used to identify the parameters of general Prony signals.

\begin{theorem}[Prony signal identification with Bernoulli schemes]
\label{thm:pronbern}
  Let $V$ be a vector space and $A : V \to V$ be a generator with active spectrum $\Lambda := \{ \lambda_1, \ldots, \lambda_M \} \subset \mathbb{C}$ of the signal space $\cM := \cM(A,\Lambda)$ as defined in~\eqref{eq:ma}.  Let $f \in \cM$, i.e.,
  \begin{equation*}
     f = \sum_{k=1}^{M} c_k v_{\lambda_k}.
  \end{equation*}
  Let furthermore ${(\mathcal{F}_m)}_{m \ge 0} \in \ell_{\cM^*}$ be an evaluation scheme satisfying the conditions of Lemma~\ref{thm:samp}. Then a GB scheme as described in Theorem~\ref{thm:gbern} can be used to recover the nonlinear parameters in $\Lambda$.
\end{theorem}

\begin{proof}
    The proof follows immediately from Theorem~\ref{thm:ratpron}. 
    Consider the sequence of functionals $\mathcal{F} := (\mathcal{F}_0, \mathcal{F}_1, \mathcal{F}_2, \ldots)$ and the rational function
    \begin{equation*}
        G := \mathcal{Z}_{w}[\mathcal{F}f],
    \end{equation*}
    where the components of $\mathcal{F}$ satisfy the conditions of Lemma~\ref{thm:samp} and $\mathcal{Z}_{w}$ is defined according to Eq.~\eqref{eq:zw} with $w > \max_{k=1,\ldots,M} |\varphi(\lambda_k)|$. Clearly, $G \in \mathcal{H}_{\varphi(\Lambda)/w} \subset H_2(\mathbb{D})$. Therefore, for any choice of $\boldsymbol{a} \in \mathbb{D}^p \ (p \in \mathbb{N})$, the TM-Fourier coefficients $\langle G, \Phi_n^{\boldsymbol{a}} \rangle_{H_2(\mathbb{D})}$ exist and we may compute the limit in Eq.~\eqref{eq:genbern}.
\end{proof}

\begin{remarks}
\begin{enumerate}[label = \alph*)]
    \item We note that using the TM-Fourier coefficients, the limit in Eq.~\eqref{eq:genbern} always exists, except if we choose the elements of the generating vector $\boldsymbol{a} \in \mathbb{D}^p$ from a certain zero measure set in~$\mathbb{D}$~\cite{SS1}.

    \item As explained in section~\ref{sec:genBer} (see also~\cite{SBS, SSSAB} and~\cite{siamdozsa}), the Bernoulli scheme can be applied multiple times to recover every nonlinear parameter in $\Lambda$.
\end{enumerate}
\end{remarks}

\section{Examples}
\label{sec:ex}

In this section we present some numerical experiments and examples related to the application of the proposed rational Prony and Bernoulli frameworks for various parameter recovery problems. 

\subsection{Linear parameter recovery with Vandermonde and Takenaka-Malmquist matrices}
\label{sec:exp1}

We empirically verify, that the Takenaka-Malmquist expansion based upper triangular problem to recover the linear parameters of a Prony signal (Theorem~\ref{thm:lin}) is indeed better conditioned than the usual Vandermonde-based formulations (see Eqs.~\eqref{eq:Vander} and~\eqref{eq:residVander}). Since here we are only interested in comparing the condition numbers of the corresponding matrices, for each experiment we assume the following:
\begin{enumerate}[label = \alph*)]
    \item The considered Prony signal is the transfer function of a discrete-time BIBO stable and causal SISO LTI system and can be written as
    $$
        H := \sum_{k=1}^M c_k r_{\lambda_k} \quad (\lambda_k \in \mathbb{D}, \ c_k \in \mathbb{C}, \ k=1,\ldots,M),
    $$
    where $M \in \mathbb{N}$ is given and $r_\lambda$ are defined according to Eq.~\eqref{eq:def_rlambda}.

    \item The nonlinear parameters $\lambda_k$ are assumed to be known for $k=1,\ldots,M$. This is assumed to make a fair comparison between the condition numbers of the proposed triangular matrix in Eq.~\eqref{eq:lintrig} and the Vandermonde formulation in~\eqref{eq:residVander}. 
\end{enumerate}

In our first experiment, we consider the (discrete-time) SISO LTI system proposed in~\cite{SSSAB}, which describes the dynamics of a flexible wing aircraft. This type of aircraft is characterized by the so-called Body Freedom Flutter (BFF) phenomenon, which causes elastic wing deformations resulting in fluttering. In~\cite{SSSAB}, a Laguerre expansion-based variant of the GB method is used to identify dominant nonlinear parameters of a transfer function obtained from real measurements describing such a system. 
The transfer function is characterized by $M=15$ nonlinear parameters in $\mathbb{D}$. 


In the second experiment we consider the building benchmark model from the SLICOT Benchmark Examples for Model Reduction~\cite{slicot}. This system describes the displacement of a multi-story building during seismic activity. It is frequently used as a benchmark model to verify LTI model order reduction methods~\cite{morAntSG01}. The continuous-time system's transfer function is described by $M=48$ nonlinear parameters. In order to make use of Theorem~\ref{thm:lin}, we discretize the system. Discretization is carried out using MATLAB's \texttt{c2d} function with the zero-order hold option, which assumes that the inputs are piecewise constant functions over the sampling time. In the experiment presented, this sampling time is chosen as $T=0.01$ which yields a stable discretized system. 

Finally, we consider a benchmark all pass system introduced in~\cite{morwiki_allpass} whose transfer function satisfies $\left|H(\ri \omega) \right| = 1 \ (\omega \in \R)$. Moreover, the zeros and poles of the transfer function are mirror images of each other across the imaginary axis. The system is discretized as the building movement example but with the sampling time $T=0.1$. The corresponding transfer function coincides with an $M$-term Blaschke product defined by multiplying exactly $M$ Blaschke factors defined in Eq.~\eqref{eq:Blaschfac}. For this experiment, we use $M=200$ and follow the zero-pole arrangement proposed in the benchmark~\cite{morwiki_allpass}.

The nonlinear parameters of the considered (discrete-time) systems are illustrated in Fig.~\ref{fig:VanderPoles}. The condition numbers of the matrices corresponding to the Vandermonde and the proposed triangular matrices are recorded in Table~\ref{tab:conds}.



Based on the above experiments, we conclude that the triangular formulation is very beneficial  when there is a large number of nonlinear parameters close to the unit circle. In fact, the benefit of the proposed formulation seems to become more pronounced as the complexity of the Prony signal (expressed by the number $M$) increases. These problems cover a large class of real-world applications. Finally, we note that Takenaka-Malmquist systems can also be defined on the open right complex half-plane (see, e.g.,~\cite{eisner2014discrete}). Because of this, recovering the linear parameters of continuous-time LTI systems can also be done with triangular matrices of the form shown in Eq.~\eqref{eq:lintrig}. Indeed, one can replace the inner products Eq.~\eqref{eq:lintrig} with inner products involving TM functions defined on the open right complex half-plane. Thus, for this particular class of problems, the discretization step used in the above examples can be omitted. Nevertheless, the proposed methodology can be applied to \textit{any} Prony signal, not just rational transfer functions of continuous-time stable LTI systems. That is, transforming the Prony signal into a finite dimensional subspace of $H_2(\mathbb{D})$, then using our GROP or GB  schemes to recover the nonlinear parameters followed by an application of Theorem~\ref{thm:lin} to find the linear parameters can be carried out for any general Prony problem.  Hence, the above example regarding the building model includes the additional discretization step.
\begin{figure}[tb]
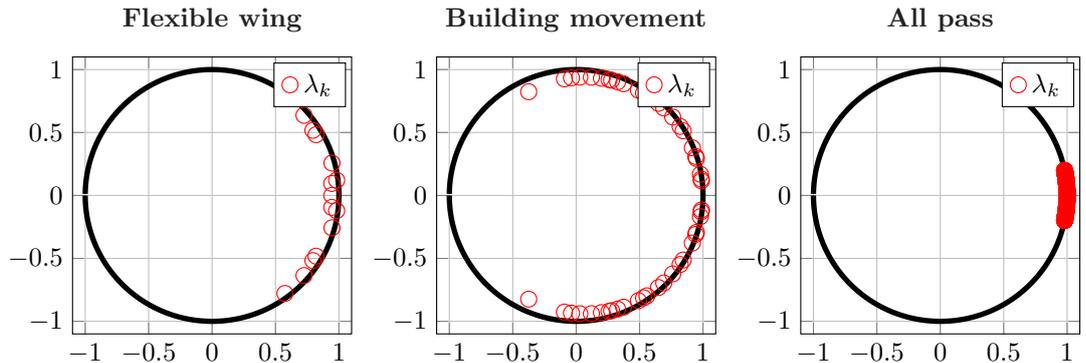

    \centering
    \input{TikZ/Flexiwing.tikz} 
    \input{TikZ/Building.tikz} 
    \input{TikZ/allpass.tikz}
    \caption{Nonlinear parameters (mirror image poles) describing the considered systems.}
    \label{fig:VanderPoles}
\end{figure}

\begin{table}[bt]
    \centering
    \caption{Spectral norm condition numbers of the system matrices for linear parameter recovery.}
    \label{tab:conds}
    \begin{tabular}{cccc}
    \toprule
        {Experiment} & Flexible wing & Building movement & All pass \\ \midrule
        {Vandermonde} & $1.6 \cdot 10^{11}$ & $2.05 \cdot 10^{17}$ & $2.08 \cdot 10^{20}$ \\ 
        {TM triangular (proposed)} & $13.53$ & $139.57$ & $135.58$ \\ \bottomrule
    \end{tabular}
\end{table}
\subsection{Pole identification of time-delayed continuous-time LTI systems}

In this section we consider an example application. Namely, we attempt to solve an identification problem of a continuous-time SISO LTI system with an unknown time delay $\tau$ in the input, i.e., in state-space form we have
\begin{align*}
  \dot{x}(t) &= Ax(t) + Bu(t-\tau), \\
         y(t) &= Cx(t) + Du(t-\tau),
\end{align*}
where the matrix $A$, the column vector $B$, the row vector $C$, and the scalar $D$ are of compatible dimensions.
It is well known, that the transfer function $H$ of a BIBO stable continuous-time system belongs to $H_2(\mathbb{C}_{+})$, where $\mathbb{C}_{+}$ denotes the open right complex half-plane. If no time delay is present, then the transfer function is rational and can be written as
\begin{equation}
    \label{eq:trfcont}
    H(s) := \sum_{k=1}^{M} \frac{c_k}{s - \lambda_k}, 
\end{equation}
where $c_k \in \mathbb{C}$ and with the  nonlinear parameters $\Lambda := \{\lambda_1, \lambda_2, \ldots, \lambda_M \} \subset \mathbb{C}_{-}$, where $\mathbb{C}_{-}$ denotes the open left complex half-plane. If the system contains a time delay $\tau > 0$, then the transfer function $H$ ceases to be rational and attains the form
\begin{equation}
    \label{eq:trfDelayed}
    H_{\tau}(s) := \e^{-\tau s} \sum_{k=1}^{M} \frac{c_k}{s - \lambda_k}. 
\end{equation}
Let $\cM$ denote the space of all signals adhering to~\eqref{eq:trfDelayed}. In our example we shall assume knowledge of $H_{\tau}(\ri\omega)$ for all $ \omega \in \mathbb{R}$ (or at least a sufficiently fine discrete sampling of it). Our primary objective is to use the proposed GROP and GB methods to recover the nonlinear parameters in $\Lambda$. To this end, we shall first construct an intuitive evaluation scheme adhering to the form presented in Lemma~\ref{thm:samp}. Once the evaluation scheme ${(\mathcal{F}_m)}_{m \ge 0}$ has been generated, we shall use it to construct a rational function in $H_2(\mathbb{D})$ whose mirror image poles can be used to recover the nonlinear parameters in $\lambda_k \ (k=1,\ldots,M)$. This is in accordance with the findings presented in Theorem~\ref{thm:ratpron}. We emphasize that the purpose of this example is to demonstrate the usefulness of GROP and GB schemes even when the Prony signal (in Eq.~\eqref{eq:trfDelayed}) is not a rational function. The development of  identification schemes based on the proposed methodology would require a much more careful analysis which falls outside of the scope of the current study. 

We consider now the following evaluation scheme ${(\mathcal{F}_m)}_{m \ge 0}$ with
\begin{equation}
    \label{eq:delayedEval}
    \mathcal{F}_m H := \frac{1}{2 \pi} \int_{-\infty}^{\infty} H(\ri\omega) \e^{\ri \omega m} \,\d\omega \quad (m \in \mathbb{N} \cup \{0\}).
\end{equation}
For any $H_{\tau} \in \cM$, the evaluation functionals from Eq.~\eqref{eq:delayedEval} are well-defined. Furthermore, we have
\begin{equation}
    \label{eq:conresp}
    \frac{1}{2 \pi} \int_{-\infty}^{\infty} H_{\tau}(\ri \omega) \e^{\ri \omega t} \,\d \omega = h_{\tau}(t),
\end{equation}
where the delayed impulse response $h_{\tau}$ is defined according to
\begin{equation*}
    h_{\tau}(t) = \sum_{k=1}^{M} c_k \e^{\lambda_k (t - \tau)} u(t - \tau),
\end{equation*}
where $u$ denotes the Heaviside function. This is a consequence of the time shift property associated with the Laplace transform (see, e.g.,~\cite[Eq.~(9.87)]{sigsys}). Using Eq.~\eqref{eq:conresp} we find that for each $m \in \mathbb{N} \cup \{0\}$,
\begin{align}
    \label{eq:imprespdel}
    \begin{split}
    \mathcal{F}_m H_{\tau} &= \sum_{k=1}^{M} c_k \e^{\lambda_k (m - \tau)} u(m - \tau) \\ &=  \sum_{k=1}^{M} c_k \e^{-\lambda_k \tau} \e^{\lambda_k m} u(m - \tau) = \sum_{k=1}^{M} \tilde{c}_k \e^{\lambda_k m} u(m - \tau), 
    \end{split}
\end{align}
where $\tilde{c}_k := c_k \e^{-\lambda_k \tau} \ (k=1,\ldots,M)$. In other words, the evaluation scheme ${(\mathcal{F}_m)}_{m \ge 0}$ samples the delayed impulse response $h_{\tau}$. Since we assume $\tau \in \mathbb{R}$, there exists an $m_0 \in \mathbb{N}$ for which $u(m - \tau) = 1 \ (m \geq m_0)$. Notice also that since $\Re( \lambda_k) < 0 \ (k=1,\ldots,M)$, we have $\e^{\lambda_k} \in \mathbb{D}\ (k=1,\ldots,M)$. Thus, we can consider the sequence
\begin{equation}
\label{eq:gseq}
    g := (0, \mathcal{F}_{m_0}H_{\tau}, \mathcal{F}_{2 m_0} H_{\tau}, \ldots ) \in \ell_2
\end{equation}
and define the notation
\begin{equation}
    \label{eq:discPolesDelaySys}
    \alpha_k := \e^{m_0 \lambda_k} \quad (k=1,\ldots,M).
\end{equation}
We shall henceforth assume that $- \pi < \Im (m_0 \lambda_k) \le \pi \ (k=1,\ldots,M)$. 
The $\mathcal{Z}$-transform of $g$ yields
\begin{equation}
    \label{eq:Gdel}
    G(z) := \sum_{m=0}^{\infty} \sum_{k=1}^{M} \tilde{c}_k \alpha_k^m z^m - \sum_{k=1}^M \tilde{c}_k = \sum_{k=1}^{M} \frac{\tilde{c}_k}{1 - {\alpha}_k z} - \sum_{k=1}^M \tilde{c}_k \quad (z \in \overline{\mathbb{D}}).
\end{equation}
Note that the transfer function $G$ does not only have the mirror image poles $\alpha_k \ (k=1,\ldots,M)$, but due to the constant term, there is an additional mirror image pole at $0$. Since $0$ cannot be written as $\e^{m_0\lambda}$ for some $\lambda \in \mathbb{C}_-$ as in Eq.~\eqref{eq:discPolesDelaySys}, this mirror image pole can be safely isolated from the $\alpha_k \ (k=1,\ldots,M)$. 
Due to the injectivity of the map $\lambda \mapsto \e^{m_0\lambda}$ on the stripe $\{z \in \mathbb{C}: \ -\pi < \Im(m_0\lambda) \le \pi \}$, the values $\alpha_k \ (k=1,\ldots,M)$ can be used to recover the nonlinear parameters  $\lambda_k \ (k=1,\ldots,M)$. 


Given $g$ from Eq.~\eqref{eq:gseq}, we can apply the GOP method (see subsection~\ref{sec:prelims}) to recover the nonlinear parameters $\lambda_k \ (k=1,\ldots,M)$. Indeed, considering  Eq.~\eqref{eq:imprespdel}, we obtain
\begin{equation*}
    \left[ \funE\big(( S^{\#})^{m+j} g\big) \right]_{m, j=0}^{M-1} \left[ p_k \right]_{k=0}^{M-1} = -\left[ \funE\big(( S^{\#})^{M + k} g\big)\right]_{k=0}^{M-1},
\end{equation*}
where $\funE : \ell \to \mathbb{C}$, $(x_0,x_1,x_2,\ldots) \mapsto x_0$ is the point evaluation functional, $S^{\#}$ denotes the backward shift operator for sequences, and $p_k \ (k=0,\ldots,M-1)$ denote the algebraic coefficients of the Prony polynomial $P_{\Lambda(S^{\#})}$ (see Eq.~\eqref{eq:pronpol}). The zeros of this polynomial coincide with the nonlinear parameters to be recovered. 

According to Theorem~\ref{thm:ratpron}, we can also recover the nonlinear parameters using the GROP method reviewed in subsection~\ref{sec:pronyrat}. In this case, the result of the transformation proposed in Theorem~\ref{thm:ratpron}, when applied to the sequence $g$ coincides with $G$ from Eq.~\eqref{eq:Gdel}. Then, choosing the evaluation functional $\funF$ according to Eq.~\eqref{eq:defF} and considering the definition of the $H_2(\mathbb{D})$ shift operator $S^{\ast}$ from Eq.~\eqref{eq:adjshift}, leads to the linear equation system
\begin{equation*}
    \left[ \funF\big(( S^{\ast})^{m+j} G\big) \right]_{m, j=0}^{M-1} \left[ p_k \right]_{k=0}^{M-1} = -\left[ \funF\big(( S^{\ast})^{M + k} G\big)\right]_{k=0}^{M-1}.
\end{equation*}
Its solution  matches the coefficients of the Prony polynomial $P_{\Lambda(S^{\ast})}$ from Eq.~\eqref{eq:ratPronPol}. The zeros of this polynomial are the parameters $\alpha_k \ (k=1,\ldots,M)$ defined in Eq.~\eqref{eq:discPolesDelaySys}, from which the nonlinear parameters $\lambda_k \ (k=1,\ldots,M)$ can be recovered.

Finally, we note that, as a consequence of Theorem~\ref{thm:pronbern}, the GB scheme can also be applied to recover the nonlinear parameters of the delayed transfer function $H_{\tau}$. In this case, since $G$ from Eq.~\eqref{eq:Gdel} can be written as a sum of elementary rationals, we can apply Theorem~\ref{thm:gbern}. This method of recovering the nonlinear parameters is especially useful when $M$ is not given a priori. Indeed, the GB method recovers $\gamma$-dominant (see Eq.~\eqref{eq:gdpole}) nonlinear parameters one-by-one and does not rely on any information about $M$. Note that $G = \mathcal{Z}_{w}[g]$, where the transformation is defined according to Eq.~\eqref{eq:zw} with $w=1$. In practice, we can only approximate the $\mathcal{Z}_w$-transformation by truncating the infinite sum, however as our experimental results show, the nonlinear parameters can still be recovered with a high precision. 

In order to demonstrate the usefulness of the proposed scheme, we simulate a small delay system, where
$$
    (\lambda_1,\lambda_2,\lambda_3):= ( -0.157 + 0.359\ri, -0.157 - 0.359\ri, -2.3 )
$$
and 
$$
    (c_1,c_2,c_3) := (0.026 + 0.195\ri, 0.026 - 0.195\ri, 0.022 ).
$$
The time delay for this example is chosen as $\tau = 1.5$.
We emphasize that $\tau$ is not known to the methods presented and cannot be readily identified. The classical Prony method is applied to Eq.~\eqref{eq:imprespdel} to recover the nonlinear parameters in $\Lambda$. In addition, the GROP and GB schemes are applied to~\eqref{eq:Gdel} for similar purposes. Fig.~\ref{fig:poles} shows the nonlinear parameters of  the continuous-time transfer function $H_\tau$ and its discrete-time counterpart $G$ recovered by the methods. Interestingly, despite the additional transformations involved, for this example, the rational methods found slightly better approximations of the nonlinear parameters. 

\begin{figure}[bt]
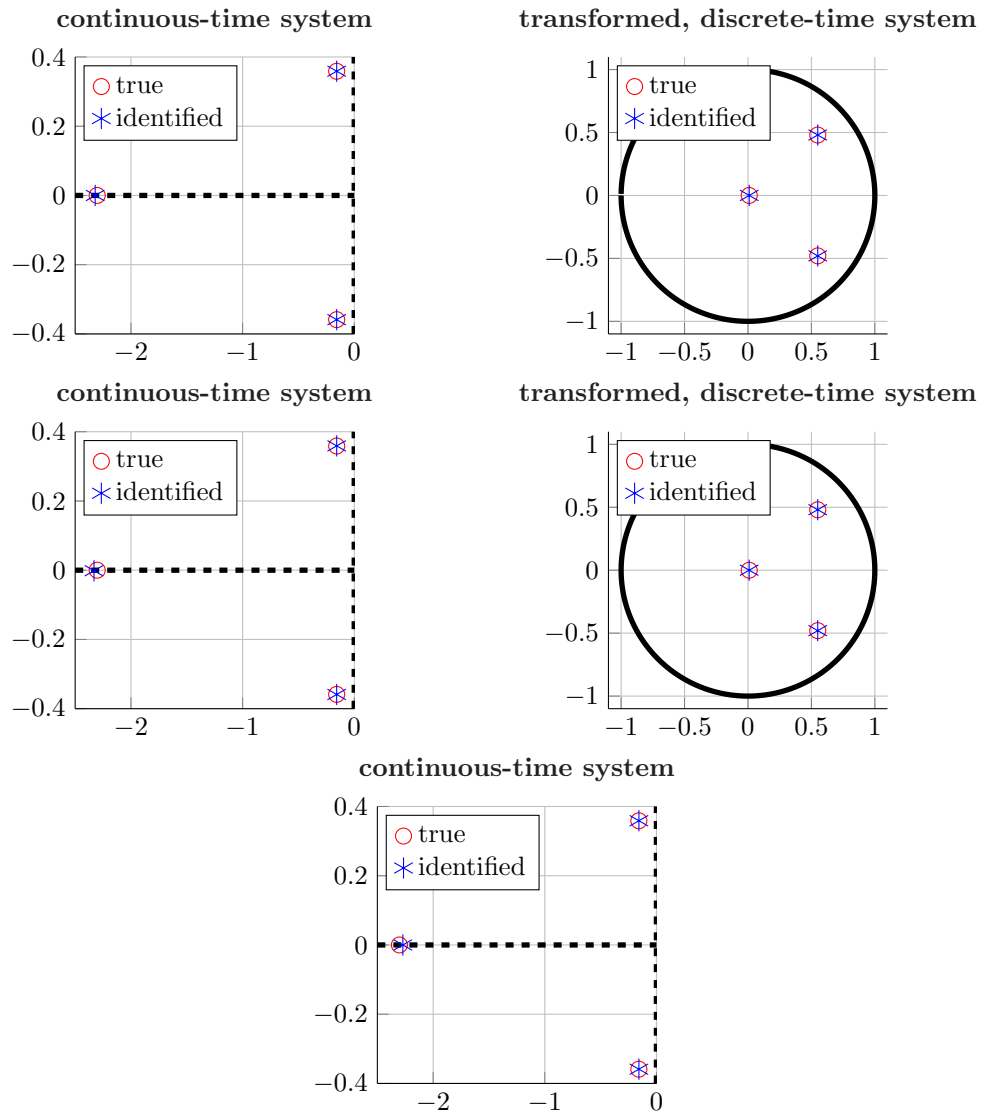

    \centering
    \input{TikZ/DelayPronyRat.tikz}
    \input{TikZ/DelayBern.tikz}
%
%
\definecolor{mycolor1}{rgb}{1.00000,0.00000,1.00000}%
\definecolor{mycolor2}{rgb}{0.12941,0.12941,0.12941}%
\begin{tikzpicture}

\begin{axis}[%
width=0.24\textwidth,
height=0.24\textwidth,
at={(6.352in,0.983in)},
scale only axis,
xmin=-2.5,
xmax=0,
ymin=-0.4,
ymax=0.4,
axis background/.style={fill=white},
title style={font=\bfseries\color{mycolor2}},
title={continuous-time system},
axis x line*=bottom,
axis y line*=left,
xmajorgrids,
ymajorgrids,
legend style={at={(0.03,0.97)}, anchor=north west, legend cell align=left, align=left}
]
\addplot [color=black, dashed, line width=2.0pt, forget plot]
  table[row sep=crcr]{%
0	-0.4\\
0	0.4\\
};
\addplot [color=black, dashed, line width=2.0pt, forget plot]
  table[row sep=crcr]{%
-2.5	0\\
0	0\\
};
\addplot [color=red, only marks, mark size=3pt, mark=o, mark options={solid, fill=red, red}]
  table[row sep=crcr]{%
-0.15735537241985	0.358770670270572\\
-0.15735537241985	-0.358770670270572\\
-2.30258509299405	0\\
};
\addlegendentry{true}

\addplot [color=blue, only marks, mark size=4pt, mark=asterisk, mark options={solid, fill=mycolor1, blue}]
  table[row sep=crcr]{%
-0.156915662864696	-0.358824370009661\\
-0.156915662864696	0.358824370009661\\
-2.27155066721228	2.2861904978423e-14\\
};
\addlegendentry{identified}

\end{axis}
\end{tikzpicture}%
    \caption{Nonlinear parameters of the transfer function of a time-delayed SISO LTI system recovered by the GROP method (top row), the GB method  (middle row) and the classical Prony algorithm (bottom row).}
    \label{fig:poles}
\end{figure}

\subsection{Recovering the parameters of reproducing kernels}

In our final example, we consider the following application of the rational methods presented here. Suppose $\mathcal{X} \subset \mathbb{C}$ is compact and let $\mathfrak{H}$ be a reproducing kernel Hilbert space (RKHS) of functions $\mathcal{X} \to \mathbb{C}$. Since $\mathfrak{H}$ is an RKHS, by Aronszajn's theorem~\cite{aronszajn1950}, there exists a unique mapping
$$
    K : \mathfrak{H} \times \mathfrak{H} \to \mathbb{C}
$$
such that for any $f \in \mathfrak{H}$ and $K_y := K(\cdot,y) \in \mathfrak{H}$,
$$
    {\langle K_y, f \rangle}_{\mathfrak{H}} = {\langle K( \cdot, y), f \rangle}_{\mathfrak{H}} = f(y) \quad (y \in \mathcal{X}).
$$
Furthermore, $K(x, y) = {\langle K_x, K_y \rangle}_{\mathfrak{H}}$ for any $x,y \in \mathcal{X}$. The RKHS property is equivalent to all point evaluation functionals on $\mathfrak{H}$ being bounded. Suppose that $\mathfrak{H}$ contains all polynomials $\mathcal{X} \to \C$. Examples of such $\mathfrak{H}$ include RKHS given by RBF kernels~\cite{jayasumana2015kernel} and those defined by Sobolev-type kernels on compact sets~\cite{narcowich2002approximation}. We are interested in problems, where the Prony signal attains the form
\begin{equation}
    \label{eq:rkhsPronySig}
    \begin{split}
        & f(x) = \sum_{k=1}^{M} c_k K_{\lambda_k}(x) = \sum_{k=1}^{M} c_k K(x, \lambda_k) \quad (x \in \mathcal{X}, \ \Lambda := \{ \lambda_1, \lambda_2, \ldots, \lambda_M \} \subset \mathcal{X},\  c_k \in \mathbb{C}). 
    \end{split}
\end{equation}
Exploiting the fact that $\mathcal{X}$ is compact (and thus bounded), there exists a $C\in \mathbb{R}$ be such that $|x| < C$ for all $x \in \mathcal{X}$. If we assume that $\mathfrak{H}$ has a polynomial basis, we can  consider evaluation schemes ${(\mathcal{F}_m)}_{m \ge 0}$ of the form
\begin{equation}
    \label{eq:rkhseval}
    \mathcal{F}_m f := {\langle f, p_m \rangle}_{\mathfrak{H}} \quad (f,p_m \in \mathfrak{H},\ p_m(\cdot) = (\cdot /C)^m,\ m \in \mathbb{N} \cup \{0\}). 
\end{equation}
Using Theorems~\ref{thm:ratpron} and~\ref{thm:gbern}, it is possible to apply the GB scheme to recover the linear and nonlinear parameters of $f$ in Eq.~\eqref{eq:rkhsPronySig}. A good example of a concrete RKHS where this scheme might be applicable is the following. Consider $\mathcal{X} := [-1, 1] \subset \mathbb{R}$ and the RKHS $\mathfrak{H}$ defined by the kernel
\begin{equation}
    \label{eq:kergood}
    K(x,y) := \sum_{k=0}^{\infty} \frac{2k +1}{1 + k^2} P_k(x) P_k(y),
\end{equation}
where $P_k$ denotes the $k$-th Legendre polynomial~\cite{szego}. This RKHS includes all polynomials, and thus the above proposed evaluation scheme is applicable for any $m \in \mathbb{N} \cup \{0\}$ and $C > 1$. 

In the following example, instead of the infinite sum given in Eq.~\eqref{eq:kergood}, we consider the $N$-dimensional RKHS defined by
\begin{equation}
    \label{eq:finiteKer}
    K(x, y) := \sum_{n=0}^{N} \pi_k(x) \pi_k(y) \quad (x, y \in [-1, 1]).
\end{equation}
Here $\pi_k := \sqrt{\frac{2k + 1}{2}} P_k$. Then, $\{\pi_k: \ k \in \mathbb{N} \cup \{0\}\}$ forms an orthonormal system in $L_2([-1,1])$. We note that for any $N \in \mathbb{N}$ and any $N$-element system of orthogonal polynomials (not necessarily Legendre), the Christoffel-Darboux formula (see, e.g.,~\cite{szego, gautschi, vandenhof}) provides a closed form for $K$. The RKHS defined by the kernel in Eq.~\eqref{eq:finiteKer} is the space of all  polynomials defined over $[-1, 1]$ of degree at most $N$. 

In our experiment, we choose $N=512$. In effect this means that we can only apply the evaluation scheme from Eq.~\eqref{eq:rkhseval} up to $m=512$. In this way, we can only approximate the SISO LTI impulse responses which are required to transform the problem to $H_2(\mathbb{D})$ (see Theorem~\ref{thm:ratpron}). Nevertheless, due to the geometric decay of $\mathcal{F}_m f$ for $m \to \infty$, we can still achieve meaningful results with the proposed rational methods. In particular, we consider a Prony signal $f$ defined according to Eq.~\eqref{eq:rkhsPronySig}, where the kernel is given by Eq.~\eqref{eq:finiteKer}. In this experiment, we want to showcase how the GB algorithm can be used to recover only a few, dominant parameters from $\Lambda$. In order to achieve this, we chose $M=30$ and $\Lambda$ from Eq.~\eqref{eq:rkhsPronySig} as an equidistant sampling of $[-0.9, -0.7]$, i.e., $\lambda_k := -0.9 + 0.2 \cdot \frac{k-1}{M-1} \ (k=1,\ldots,M)$. The evaluation scheme from Eq.~\eqref{eq:rkhseval} is applied to $f$ to produce the sequence $g={(g_m)}_{m \ge 0}$, which satisfies
\begin{equation*}
    g_m := \sum_{k=1}^{M}c_k (\lambda_k/C)^m \quad (m \in \mathbb{N} \cup \{0\}).
\end{equation*}
Indeed, 
\begin{equation*}
    \begin{split}
        & \mathcal{F}_m f = \mathcal{F}_m \left( \sum_{k=1}^M c_k K(\cdot, \lambda_k) \right) = \sum_{k=1}^M c_k \left\langle K(\cdot, \lambda_k), \left((\cdot)/C\right)^m \right\rangle_{\mathfrak{H}} = \sum_{k=1}^M c_k (\lambda_k/C)^m,
    \end{split}
\end{equation*}
where in the last equality we exploited the reproducing property of the kernel $K$. Similarly to the previous experiment, the GOP  method is applied to $G:= \mathcal{Z}[g]$ to recover the parameters $\lambda_k$ and $c_k$. As seen in  Fig.~\ref{fig:RKHS}, the proximity of the nonlinear parameters to be recovered poses a challenge for the  GOP method. Indeed, in our experiments the sampling matrix becomes ill-conditioned which results in poor approximations of $\lambda_k$ and $c_k$. It should also be noted, that reducing the number of assumed Prony atoms ($M$) when constructing the sampling matrix~\eqref{eq:sampmat} provides no guarantee that  GOP method recovers the actual signal parameters $\lambda_k$. Since the Prony signal $G$ is rational, the same behavior is expected for the GROP method which in this sense is equivalent to the GOP method.

On the other hand, applying the GB scheme to $G$, we can recover a reduced number of nonlinear parameters $\lambda_k$. This is well-reflected by our results in Fig.~\ref{fig:RKHS2}. In this case, the GB algorithm is able to find $M_r = 2$ nonlinear parameters of the signal nearly perfectly while the GOP method (also for $M_r = 2$) suffers from significant errors. We note that in all our experiments, we rely on the algorithms proposed in~\cite{siamdozsa} to select the free parameters of the GB method.


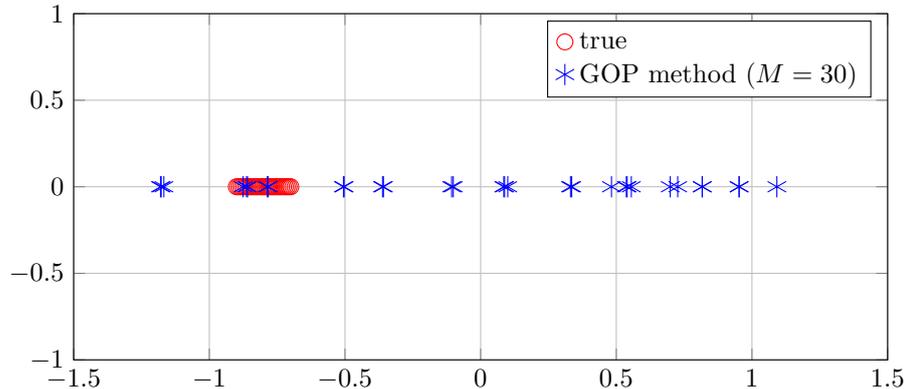
\begin{figure}[bt]
    \centering
%
%
\definecolor{mycolor1}{rgb}{0.12941,0.12941,0.12941}%
\begin{tikzpicture}

\begin{axis}[%
width=0.7\textwidth,
height=0.3\textwidth,
at={(1.454in,0.741in)},
scale only axis,
xmin=-1.5,
xmax=1.5,
ymin=-1,
ymax=1,
axis background/.style={fill=white},
title style={font=\bfseries\color{mycolor1}},
xmajorgrids,
ymajorgrids,
legend style={legend cell align=left, align=left}
]
\addplot [color=red, only marks, mark size=3pt, mark=o, mark options={solid, fill=red, red}]
  table[row sep=crcr]{%
-0.9	0\\
-0.893103448275862	0\\
-0.886206896551724	0\\
-0.879310344827586	0\\
-0.872413793103448	0\\
-0.86551724137931	0\\
-0.858620689655172	0\\
-0.851724137931035	0\\
-0.844827586206897	0\\
-0.837931034482759	0\\
-0.831034482758621	0\\
-0.824137931034483	0\\
-0.817241379310345	0\\
-0.810344827586207	0\\
-0.803448275862069	0\\
-0.796551724137931	0\\
-0.789655172413793	0\\
-0.782758620689655	0\\
-0.775862068965517	0\\
-0.768965517241379	0\\
-0.762068965517241	0\\
-0.755172413793103	0\\
-0.748275862068965	0\\
-0.741379310344828	0\\
-0.73448275862069	0\\
-0.727586206896552	0\\
-0.720689655172414	0\\
-0.713793103448276	0\\
-0.706896551724138	0\\
-0.7	0\\
};
\addlegendentry{true}

\addplot [color=blue, only marks, mark size=4pt, mark=asterisk, mark options={solid, fill=blue, blue}]
  table[row sep=crcr]{%
-1.16746424751663	-0\\
-1.17983630091943	-0\\
-1.17722628235984	-0\\
-0.859003275023961	-0\\
-0.782582675035618	-0\\
-0.875738758304668	-0\\
-0.862880206455697	-0\\
-0.786625951523224	-0\\
-0.505706037868703	-0\\
-0.358188991710682	-0\\
-0.106906539707381	-0\\
0.0883310684205315	0\\
0.331142572560632	0\\
0.539936768978086	0\\
-0.503143532166259	-0\\
-0.361636147915902	-0\\
-0.10080831061842	-0\\
0.0852238335954886	0\\
0.335693269673786	0\\
0.536893646566699	0\\
0.817005078841336	0\\
0.816305806089429	0\\
1.09186533578887	0\\
0.952989902657889	0\\
0.952285263810996	0\\
0.726327654590445	0\\
0.482227004947461	0\\
0.698832582647303	0\\
0.555390781839929	0\\
0.1	0\\
};
\addlegendentry{GOP method ($M=30$)}

\end{axis}

\end{tikzpicture}%
    \caption{Nonlinear parameters $\lambda_k$ recovered by the GOP  method. The proximity of the parameters to be recovered can cause the Prony sampling matrix to become ill-conditioned which results in poor nonlinear parameter recovery.}
    \label{fig:RKHS}
\end{figure}

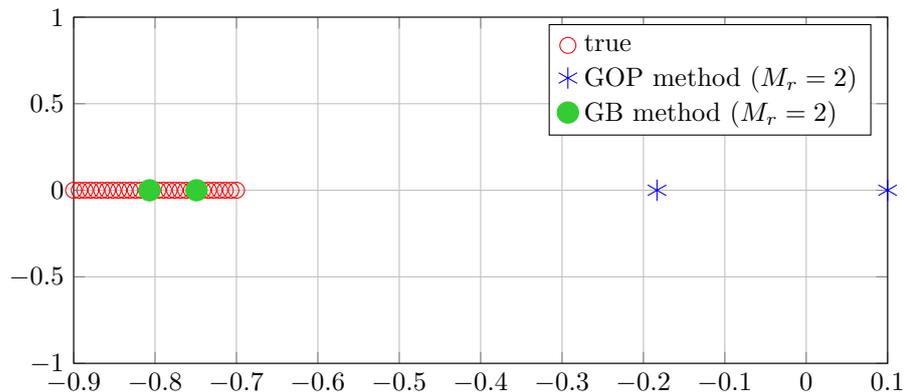
\begin{figure}[bt]
    \centering
%
%
\definecolor{mycolor1}{rgb}{0.12941,0.12941,0.12941}%
\definecolor{mygreen}{RGB}{50,205,50}
\begin{tikzpicture}

\begin{axis}[%
width=0.7\textwidth,
height=0.3\textwidth,
at={(1.454in,0.301in)},
scale only axis,
xmin=-0.9,
xmax=0.1,
ymin=-1,
ymax=1,
axis background/.style={fill=white},
title style={font=\bfseries\color{mycolor1}},
xmajorgrids,
ymajorgrids,
legend style={legend cell align=left, align=left}
]
\addplot [color=red, only marks, mark size=3pt, mark=o, mark options={solid, fill=red, red}]
  table[row sep=crcr]{%
-0.9	0\\
-0.893103448275862	0\\
-0.886206896551724	0\\
-0.879310344827586	0\\
-0.872413793103448	0\\
-0.86551724137931	0\\
-0.858620689655172	0\\
-0.851724137931035	0\\
-0.844827586206897	0\\
-0.837931034482759	0\\
-0.831034482758621	0\\
-0.824137931034483	0\\
-0.817241379310345	0\\
-0.810344827586207	0\\
-0.803448275862069	0\\
-0.796551724137931	0\\
-0.789655172413793	0\\
-0.782758620689655	0\\
-0.775862068965517	0\\
-0.768965517241379	0\\
-0.762068965517241	0\\
-0.755172413793103	0\\
-0.748275862068965	0\\
-0.741379310344828	0\\
-0.73448275862069	0\\
-0.727586206896552	0\\
-0.720689655172414	0\\
-0.713793103448276	0\\
-0.706896551724138	0\\
-0.7	0\\
};
\addlegendentry{true}

\addplot [color=blue, only marks, mark size=4pt, mark=asterisk, mark options={solid, fill=blue, blue}]
  table[row sep=crcr]{%
-0.183333333333333	-0\\
0.1	0\\
};
\addlegendentry{GOP method ($M_r=2$)}

\addplot [color=mygreen, only marks, mark size=4pt, mark=*, mark options={solid, fill=mygreen, mygreen}]
  table[row sep=crcr]{%
-0.80675234955849	-0\\
-0.749022666341454	0\\
};
\addlegendentry{GB method ($M_r=2$)}

\end{axis}
\end{tikzpicture}%
    \caption{The GB  method can find the actual nonlinear parameters in an iterative fashion. When attempting to reduce the size of the Prony sampling matrix (i.e., by reducing $M$), there are no guarantees that the GOP method will find the actual nonlinear parameters.}
    \label{fig:RKHS2}
\end{figure}

\section{Conclusions}
\label{sec:conc}

As a main contribution of this study we showed that parameter recovery problems which can be solved using the generalized operator-based Prony method always corresponds to solving SISO LTI identification tasks. Indeed, given access to a Prony signal we can always find a transformation, which maps the Prony signal to an appropriate finite-dimensional subspace of $H_2(\mathbb{D})$. This subsequently allowed rational identification algorithms to be applied to general Prony problems. In particular, we considered the application of generalized Bernoulli schemes. We gave examples where the parameters of irrational Prony signals could be successfully recovered using rational Prony and Bernoulli schemes. In addition we showed that treating such problems in a rational setting has practical benefits as certain numerical problems (e.g., solving large Vandermonde systems) can be overcome.

In our future work, we would like to demonstrate the benefit of using the considered approach (especially the generalized Bernoulli scheme) on real-world identification problems. To this end, we shall consider the problem of more efficiently optimizing the parameters (i.e., the TM system parameters) of the Bernoulli algorithm. Furthermore, we shall investigate the effect of measurement noise and propose theoretically sound ways to mitigate it.

\section*{Code and data availability}
The MATLAB implementation of the proposed methods and experiments can be downloaded from 
\begin{center}
\url{https://gitlab.com/tamasdzs/genRatPronBern}.
\end{center}

\section*{CRediT author statement}
\begin{itemize}
    \item \textbf{Tam\'as D\'ozsa}: Conceptualization, Methodology, Software, Writing - Original Draft
    \item \textbf{Matthias Voigt}:  Writing - Review \& Editing, Supervision 
    \item \textbf{Zolt\'an Szab\'o}: Investigation, Resources
    \item \textbf{J\'ozsef Bokor}: Conceptualization, Project administration
    \item \textbf{P\'eter Kov\'acs}: Funding acquisition
\end{itemize}

\section*{Acknowledgment}

This project has received funding from the Swiss Government Excellence Scholarship No. 2025.0057. The research was supported by the Hungarian National Research, Development and Innovation Office in the framework of project MEC\_R\_149388. The research was supported by the European Union within the framework of the National Laboratory for Autonomous Systems  Research, Development and Innovation Fund (RRF-2.3.1-21-2022-00002), financed under the the K\_23 ``OTKA" funding schemes and the University Excellence Fund of Eötvös Loránd University, Budapest, Hungary (ELTE).

\bibliographystyle{plain}
\bibliography{refs_formatted}

\end{document}